\documentclass{amsart}
\usepackage[top=1.4in, bottom=1.2in, left=1.2in, right=1.2in]{geometry}
\usepackage{xcolor}
\usepackage{fancyhdr}
\usepackage{amssymb,amsfonts,amsmath,amsthm,tikz,tikz-cd,graphicx,url,mathtools,thmtools,enumerate,comment}
\usepackage[all,arc]{xy}
\usepackage{mathrsfs}
\usepackage[colorlinks=true]{hyperref}
\usepackage[alphabetic]{amsrefs}
\usepackage[sc,osf]{mathpazo}
\usepackage{subfigure}
\usepackage{upgreek}
\usetikzlibrary{matrix}
\usepackage[T1]{fontenc}
\usepackage{dbnsymb}
\usepackage[utf8]{inputenc}
\usepackage{optparams,eurosym}

\newtheorem{thm}{Theorem}[section]
\newtheorem{cor}[thm]{Corollary}
\newtheorem{prop}[thm]{Proposition}
\newtheorem{lem}[thm]{Lemma}
\newtheorem{conj}[thm]{Conjecture}
\newtheorem{quest}[thm]{Question}

\theoremstyle{definition}
\newtheorem{defn}[thm]{Definition}

\newtheorem{con}[thm]{Construction}
\newtheorem{exmp}[thm]{Example}

\theoremstyle{remark}
\newtheorem{rem}[thm]{Remark}

\newtheorem{warn}[thm]{Warning}

\newcommand{\Z}{\mathbb{Z}}

\newcommand{\ZZ}{\mathbb{Z}/2}

\newcommand{\RR}{\mathbb{R}}

\newcommand{\FF}{\mathbb{F}}

\newcommand{\cag}{\mathcal{G}}

\newcommand{\caf}{\mathcal{F}}

\DeclareMathOperator{\idd}{id}

\DeclareMathOperator{\Cone}{Cone}
\DeclareMathOperator{\CKhI}{CKhI}
\DeclareMathOperator{\KhI}{KhI}
\DeclareMathOperator{\Kh}{Kh}
\DeclareMathOperator{\BN}{BN}
\DeclareMathOperator{\CBN}{CBN}
\DeclareMathOperator{\CBNI}{CBNI}

\DeclareMathOperator{\CKh}{CKh}

\DeclareMathOperator{\HFI}{\widehat{HFI}}
\DeclareMathOperator{\HF}{\widehat{HF}}
\DeclareMathOperator{\Br}{Br}

\title{The flip map and involutions on Khovanov homology}

\author{Daren Chen}
\email{darenc@caltech.edu}
\address{Department of Mathematics, Caltech, Pasadena, CA 91125}

\author{Hongjian Yang}
\date{}
\email{yhj@stanford.edu}
\address{Department of Mathematics, Stanford University, Stanford, CA 94305}

\begin{document}
	
\begin{abstract}
The flip symmetry on knot diagrams induces an involution on Khovanov homology. We prove that this involution is determined by its behavior on unlinks; in particular, it is the identity map when working over $\mathbb{F}_2$. This confirms a folklore conjecture on the triviality of the Viro flip map. As a corollary, we prove that the symmetries on the transvergent and intravergent diagrams of a strongly invertible knot induce the same involution on Khovanov homology. We also apply similar techniques to study the half sweep-around map.
\end{abstract}

\maketitle

\setcounter{tocdepth}{1}
\tableofcontents

\setlength{\parskip}{0.25\baselineskip}

\section{Introduction}
\subsection{The flip symmetry}

The first step of defining Khovanov homology \cite{khovanov2000categorification} is to choose a diagram of the link, which can be compared to the choice of Heegaard diagram when defining Heegaard Floer homology \cite{ozsvath2004holomorphic}. The choices involved here lead to additional symmetries. In the setting of Heegaard Floer homology, Hendricks and Manolescu \cite{hendricks2017involutive} utilized the \textit{conjugation symmetry} on Heegaard diagrams to  define involutive Heegaard Floer homology.  

The present paper explores the \textit{flip symmetry}, an analog of the conjugation symmetry, in the setting of Khovanov homology. Recall that a \textit{diagram} of a link $K$ consists of a representative in the isotopy class of $K$ and a choice of a plane $P$ such that the only singularities of the projection are transverse double points. This plane is supposed to be \textit{oriented} to encode the information of over and under strands. It is clear that if $P$ admits a diagram for $K$, then so does $\overline{P}$, the same plane with the orientation reversed, and one has an obvious way to identify these two planes. This motivates the following definition.

\begin{defn}
Let $D$ be a diagram of $K$. The \textit{flip diagram} $D^*$ is the diagram of $K$ obtained by reflecting $D$ across an axis in the plane, and then reversing all crossings (i.e., switching the undercrossings with overcrossings). Equivalently, $D^*$ is the diagram of $K$ obtained by rotating $K$ in $\RR^3$ around an axis in the plane by $\pi$ and projecting it onto the same plane.
\end{defn}

To our knowledge, Viro first suggested the idea of the flip symmetry, which Lipshitz and Sarkar explored to prove that their Khovanov stable homotopy type is independent of the choice of ladybug matchings \cite[Proposition 6.5]{lipshitz2014khovanov}. 

The conjugation symmetry is a key ingredient in constructing the $\iota$ map on Heegaard Floer chain complexes \cite{hendricks2017involutive}. In our setting, the flip symmetry leads to a similar construction on Khovanov chain complex, as follows. There is a canonical identification between resolutions of $D$ and $D^*$ by simply flipping all the circles in the resolution. This gives a canonical identification between the Khovanov chain complexes of $D$ and $D^*$:\[\eta\colon \CKh(D)\to\CKh(D^*).\]On the other hand, there is a preferred isotopy between $D$ and $D^*$, given by rotating $K$ by $\pi t\,(t\in[0,1])$ around an axis on the plane counterclockwise. After fixing a version of Khovanov homology with full functoriality (e.g. by \cites{jacobsson2004invariant,khovanov2002functor,bar2005khovanov} when working over $\FF=\FF_2$, or by \cite{blanchet2010oriented,sano2021fixing} when working over $\Z$), there is a chain homotopy equivalence \[R\colon \CKh(D^*)\to\CKh(D)\]induced by a sequence of Reidemeister moves that presents the link cobordism given by the trace of the preferred isotopy, well-defined up to chain homotopy. While the flip map can be defined over $\Z$, we work with $\FF=\FF_2$ coefficients throughout this paper, unless otherwise specified. 

\begin{defn}
	Let $D$ be a diagram of $K$. The \textit{flip map} is defined as the composition of the two maps discussed above: \[\kappa\colon\CKh(D)\xrightarrow{\eta}\CKh(D^*)\xrightarrow{R}\CKh(D).\]
\end{defn}

Figure \ref{fig trefoil kappa} provides an example of the flip map on the right-handed trefoil. An explicit sequence of Reidemeister moves that realizes the trace of the rotation for the unknot is given in Figure \ref{fig unknot kappa}.

    \begin{figure}[ht]
	\[{
		\fontsize{8pt}{10pt}\selectfont
		\def\svgscale{1}
		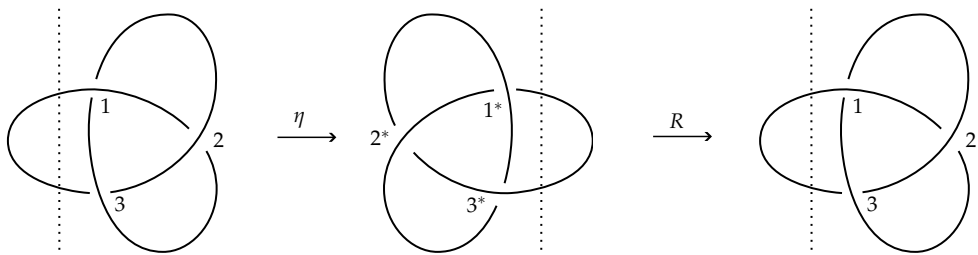
	}\]
	\caption{The flip map on the right-handed trefoil. The dashed line is the rotational axis. The labels $1^*,2^*,3^*$ are the crossings corresponding to $1,2,3$ respectively under the map $\eta$.}
    \label{fig trefoil kappa}
\end{figure}

\begin{warn}
    At first glance, the first and the second diagrams in Figure \ref{fig trefoil kappa} seem to be the same (up to planar isotopy), reflecting the fact that they are transvergent diagrams of a strongly invertible knot; see Section \ref{sec: si knots} for further discussion. However, one cannot simply use a planar isotopy as the required Reidemeister moves in the definition of $R$, since it represents a different link cobordism (cf. \cite{jacobsson2004invariant}). 
\end{warn}

Additional information from the $\iota$ map in Heegaard Floer homology has led to remarkable topological applications, for example,  \cite{dai2023infinite,dai20242}. On the other hand, a folklore conjecture asserts that the flip map on Khovanov homology (at least with $\FF_2$ coefficients) is trivial. Our main theorem confirms this conjecture.

\begin{thm}\label{thm kappa=id}
	The flip map $\kappa\colon\CKh(D)\to\CKh(D)$ is chain homotopic to the identity map on $\CKh(D)$.
\end{thm}

A corollary of Theorem \ref{thm kappa=id} is that involutive Khovanov homology, defined as an analog of involutive Heegaard Floer homology in Section \ref{sec:KhI}, does not provide new information than the ordinary Khovanov homology; see Corollary \ref{coro: KhI trivial}.

\subsection{Algebraic and topological identifications of Khovanov chain complexes}

The study of the flip map essentially boils down to the question of comparing the maps $\eta$ and $R^{-1}$ from $\CKh(D)$ to $\CKh(D^*)$. In general, when we have two diagrams $D$ and $D'$ of the same link $K$ that are related by certain planar constructions, such as flipping the diagram, or Constructions \ref{cons sweep around} and \ref{cons braid rotation} below, there are typically two approaches to identity the chain complexes $\CKh(D)$ and $\CKh(D')$:

\begin{itemize}
    \item The \textit{algebraic} identification: depending on the construction of $D'$, there would be an obvious one-to-one correspondence between crossings in $D$ and $D'$, and circles in the corresponding resolutions. This induces a map between $\CKh(D)$ and $\CKh(D')$ by simply identifying generators with the ones in the corresponding resolutions.

    \item The \textit{topological} identification: depending on the construction of $D'$, there would be a preferred cobordism between $D$ and $D'$. This cobordism induces a map between $\CKh(D)$ and $\CKh(D')$ by presenting the cobordism as a movie presentation between $D$ and $D'$. It is only well-defined up to chain homotopy.
\end{itemize}

For the flip construction, $\eta$ is the algebraic identification, and $R^{-1}$ is the topological identification.

\begin{con}\label{cons sweep around}
    Let $T$ be the diagram of a $1$-$1$ tangle, presented vertically. We have two choices to close it up on the plane: from the left and from the right. These two closures are related by a preferred cobordism formed by the trace of rotating the strand from one side to the other\footnote{Technically speaking, there are two choices: sweep from the top or from the bottom; we just choose (and fix thereafter) one of them. A posteriori, \cite{morrison2022invariants}*{Theorem~1.1} asserts that these two cobordism maps are chain homotopic. Theorem \ref{thm alg=top 11 tangle} does not rely on their theorem.}; see the left part of Figure \ref{fig:braid flip}. The topological identification in this case is called the \textit{half sweep-around map}, which was used by Morrison, Walker, and Wedrich to prove the functoriality of Khovanov homology for links in $S^3$ \cite{morrison2022invariants}.
\end{con}

\begin{con}\label{cons braid rotation}
    Let $\Br_n$ be the braid group on $n$ strands with standard generators $\sigma_i$ ($1\le i\le n-1$). There is an involution on $\Br_n$ given by sending $\sigma_i^{\pm 1}$ to $\sigma_{n-i}^{\pm 1}$ ($1\le i\le n-1$). For $\beta\in \Br_n$, we define its \textit{flip braid} $\beta^*$ as the image of $\beta$ under this involution. The closures $\widehat{\beta}$ and $\widehat{\beta^*}$ give two diagrams of the same link. They are related by a ``rolling the braid'' cobordism; see the right part of Figure \ref{fig:braid flip}. 
\end{con}

\begin{figure}[h!]
	\[{
		\fontsize{8pt}{10pt}\selectfont
		\def\svgscale{1.2}
		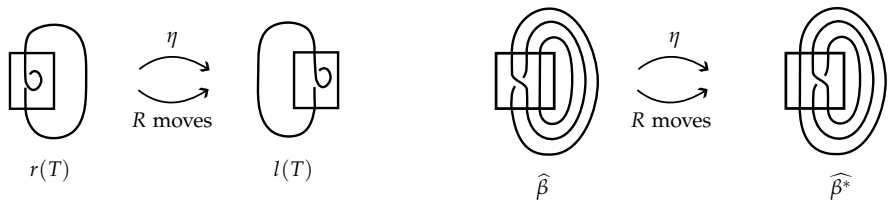
	}\]
	\caption{Left: maps between two closures of a $1$-$1$ tangle (left). Right: maps from $\widehat{\beta}$ to $\widehat{\beta^*}$.}
	\label{fig:braid flip}
\end{figure}

The algebraic identification is generally much easier to compute than the topological identification since the latter involves an unspecified sequence of Reidemeister moves, which is often difficult to describe explicitly. On the other hand, the topological identification may have better available functoriality statements. Heuristically, these two identifications should coincide for suitable constructions. The following theorems make this heuristic rigorous. The first one reformulates Theorem \ref{thm kappa=id}. 

\setcounter{section}{1}
\setcounter{thm}{3}
\begin{thm}
      The algebraic and topological identifications between $\CKh(D)$ and $\CKh(D^*)$ are chain homotopic. 
\end{thm}
  \setcounter{section}{1}
\setcounter{thm}{6}

For Constructions \ref{cons sweep around} and \ref{cons braid rotation}, we have similar results.

\begin{thm}\label{thm alg=top 11 tangle}
    The algebraic and topological identifications between the Khovanov chain complexes of the left and right closures of a $1$-$1$ tangle diagram are chain homotopic. 
\end{thm}

\begin{thm}\label{thm alg=top braid}
     The algebraic and topological identifications between the Khovanov chain complexes of the closures of a braid and its flip braid are chain homotopic. 
\end{thm}

A corollary of Theorem \ref{thm alg=top 11 tangle} is the functoriality of Khovanov homology (with $\FF_2$ coefficients) for links in $S^3$, originally established in \cite{morrison2022invariants} in much more generality; see Corollary \ref{cor: functoriality in S3}.

\subsection{A corollary for strongly invertible knots}

The triviality results in Theorems \ref{thm kappa=id}, \ref{thm alg=top 11 tangle}, and \ref{thm alg=top braid} are particularly powerful when one hopes to combine the advantages of both algebraic and topological identifications: the former is more computable whereas the latter usually has better functoriality. We illustrate this by an interesting implication for strongly invertible knots. 

Recall that a \textit{strongly invertible knot} is a knot $K$ with an orientation-preserving involution $\tau$ of $S^3$ that preserves $K$ (setwise). The fixed point set of $\tau$, or the \textit{axis}, is necessarily an unknotted $S^1$ by the resolution of the Smale Conjecture. The axis is required to intersect $K$ at two points. After removing one point from the axis, one can view $\tau$ as rotating by $\pi$ around the axis in $\RR^3$. A \textit{symmetric diagram} is a choice of the projection plane that is fixed by $\tau$. There are two types of symmetric diagrams for a strongly invertible knot: the \textit{intravergent diagram}, where the axis is perpendicular to the projection plane, and the \textit{transvergent diagram}, where the axis is on the projection plane. See Figure \ref{fig two diagrams for si knots} for illustration.

    \begin{figure}[ht]
	\[{
		\fontsize{8pt}{10pt}\selectfont
		\def\svgscale{1}
		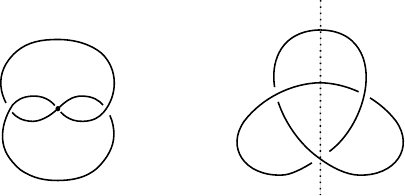
	}\]
	\caption{Intravergent (left) and transvergent (right) diagrams of the right-handed trefoil. The axis is indicated by a dot in the intravergent diagram and by a dashed line in the transvergent diagram.}
    \label{fig two diagrams for si knots}
\end{figure}

 Given a symmetric diagram of $K$, the strong inversion $\tau$ induces an involution on the corresponding Khovanov chain complex by applying $\tau$ to the cube of resolutions. Two types of symmetric diagrams induce two involutions on the Khovanov homology of $K$. It has remained an open problem whether these involutions coincide. Our next result provides an affirmative answer; see Theorem \ref{thm si knots explicit ver} for the explicit statement.

\begin{thm}\label{thm strg invertible same}
	Let $K$ be a strongly invertible knot. Then the involutions induced from the intravergent diagram and the transvergent diagram of $K$ on $\Kh(K)$ coincide.
\end{thm}

Theorem \ref{thm strg invertible same} provides further evidence to Witten's proposal of a gauge-theoretic approach to Khovanov homology \cite{witten2011fivebranes,witten2012khovanov}. In the conjectural picture, the involutions defined using two different types of diagrams should come from the same strong inversion $\tau$, so one expects they coincide\footnote{Nevertheless, as pointed out to us by Weifeng Sun, the complex symmetry breaking for the Kapustin--Witten and Haydys--Witten equations, introduced to achieve transversality, may depend on the choice of projection planes, indicating certain analytical subtleties remain in translating this prediction into rigor.}, just as in other Floer-theoretic invariants with external symmetries (see, for example, \cite{alfieri2020connected,dai2022corks,alfieri2023involutions}). However, a Floer-theoretic package from the Kapustin--Witten and Haydys--Witten equations is far from being fully developed.

\subsection{Strategy}

Proofs of Theorems \ref{thm kappa=id}, \ref{thm alg=top 11 tangle}, and \ref{thm alg=top braid} are largely inspired by works on involutive bordered Floer homology \cite{hendricks2019involutive,alishahi2023khovanov}. As an involutive refinement of bordered Floer homology \cite{lipshitz2018bordered}, involutive bordered Floer homology enables computing involutive Heegaard Floer homology by cutting-and-gluing techniques. In particular, it was proved by Alishahi--Truong--Zhang \cite[Theorem 4.1]{alishahi2023khovanov} using bordered Floer homology techniques that the $\iota$ map for the Heegaard Floer chain complex of the branched double cover of a knot $K$ can be perturbed to a filtered map with respect to the cubical grading from the cube of resolutions from a diagram of $K$.

The appropriate analog of bordered Floer homology in Khovanov homology is the tangle invariants developed independently by Khovanov  \cite{khovanov2002functor} and Bar-Natan \cite{bar2005khovanov}. Both of them are robust enough to support the arguments in this paper, and we mostly follow Khovanov's formalism. The main topological ingredient in the proof of Theorem \ref{thm kappa=id} is that the flip map $\kappa$ can be computed by introducing pairs of half twists and their inverses between crossings; see Construction \ref{con: flip map}. Using a rigidity argument, we prove our key technical result, Lemma \ref{lem algebraic flipping of an elementary tangle}, that $\kappa$ is chain homotopic to a filtered map $\kappa'$ with respect to the cubical grading. On the other hand, $\kappa$ has homological degree zero, so $\kappa'$ must \textit{preserve} the cubical grading. By explicitly computing $\kappa$ for unlinks (Example \ref{exmp unlink kappa}), we prove that it is chain homotopic to the identity map, and hence prove Theorem \ref{thm kappa=id}. The proof of Theorem \ref{thm alg=top 11 tangle} follows from a similar argument for $1$-$1$ tangles, and Theorem \ref{thm alg=top braid} is a simple combination of the results in Theorems \ref{thm kappa=id} and \ref{thm alg=top 11 tangle}.

\subsection*{Outline of the paper}

Background on Khovanov homology is provided in Section \ref{sec:bg}. In Section \ref{sec:KhI}, we use the flip map to define involutive Khovanov homology and establish some basic properties. We then prove our main theorems, Theorem \ref{thm kappa=id} in Section \ref{sec:flip map} and Theorems \ref{thm alg=top 11 tangle} and \ref{thm alg=top braid} in Section \ref{sec: half sweep around}. In Section \ref{sec: si knots}, we discuss consequences of our main theorem for strongly invertible knots and prove Theorem \ref{thm strg invertible same}. In Section \ref{sec: discussions}, we compare our invariant to involutive Heegaard Floer homology and discuss some further directions. 

\subsection*{Acknowledgments}
This project was initiated by a conversation with Qiuyu Ren, and we thank him for sharing many enlightening ideas. We have also benefited from discussions with Gary Guth, Robert Lipshitz, Yi Ni, Taketo Sano, and Sucharit Sarkar. We are particularly grateful to Robert Lipshitz for valuable comments on an earlier draft of this paper, and Taketo Sano for pointing out Proposition \ref{prop:BNI trivial} and its proof. We are deeply indebted to our advisor Ciprian Manolescu for his continued guidance and encouragement. This work was partially supported by the Simons Collaboration Grant on New Structures in Low-Dimensional Topology.

\section{Background}\label{sec:bg}

In this section, we briefly review the construction of Khovanov homology for links and tangles to set up notation for the rest of the paper. The main references are \cite{khovanov2000categorification,khovanov2002functor,khovanov2006invariant}; see also \cite{bar2002khovanov,bar2005khovanov}.

\subsection{Conventions for Khovanov homology}
We follow the conventions of Khovanov homology in \cite{bar2002khovanov}, which we briefly summarize here. Throughout the paper, we work over $\mathbb{F} = \mathbb{F}_2$, the field of two elements, unless otherwise specified. 

Let $K$ be a link in $S^3$, and $D$ be a link diagram of $K$ with $n$ crossings, of which $n_+$ are positive crossings, and $n_-$ are negative crossings. Let $V$ be the Frobenius algebra $V = \mathbb{F}[X]/(X^2)$ with multiplication $m$ induced from the polynomial algebra and comultiplication \[\Delta(1) = 1\otimes X + X\otimes 1,\, \Delta(X) = X\otimes X.\]There is a quantum grading $\deg_q$ on $V$ given by
\[\deg_q(1)=1,\, \deg_q(X)=-1.\]Fixing an ordering of the crossings, we form the standard cube \[\underline{1}^n\coloneqq(0\rightarrow1)^n\] of resolutions $D_v$ of $D$ for each vertex $v\in \underline{1}^n,$ with $0$-and $1$-smoothings defined as in Figure \ref{fig:smoothing}.
\begin{figure}[ht]
    	\[
    	{
    		\fontsize{7pt}{9pt}\selectfont
    		\def\svgscale{0.3}
    		%% Creator: Inkscape 1.0.2 (e86c870, 2021-01-15), www.inkscape.org
%% PDF/EPS/PS + LaTeX output extension by Johan Engelen, 2010
%% Accompanies image file 'smoothing.pdf' (pdf, eps, ps)
%%
%% To include the image in your LaTeX document, write
%%   \input{<filename>.pdf_tex}
%%  instead of
%%   \includegraphics{<filename>.pdf}
%% To scale the image, write
%%   \def\svgwidth{<desired width>}
%%   \input{<filename>.pdf_tex}
%%  instead of
%%   \includegraphics[width=<desired width>]{<filename>.pdf}
%%
%% Images with a different path to the parent latex file can
%% be accessed with the `import' package (which may need to be
%% installed) using
%%   \usepackage{import}
%% in the preamble, and then including the image with
%%   \import{<path to file>}{<filename>.pdf_tex}
%% Alternatively, one can specify
%%   \graphicspath{{<path to file>/}}
%% 
%% For more information, please see info/svg-inkscape on CTAN:
%%   http://tug.ctan.org/tex-archive/info/svg-inkscape
%%
\begingroup%
  \makeatletter%
  \providecommand\color[2][]{%
    \errmessage{(Inkscape) Color is used for the text in Inkscape, but the package 'color.sty' is not loaded}%
    \renewcommand\color[2][]{}%
  }%
  \providecommand\transparent[1]{%
    \errmessage{(Inkscape) Transparency is used (non-zero) for the text in Inkscape, but the package 'transparent.sty' is not loaded}%
    \renewcommand\transparent[1]{}%
  }%
  \providecommand\rotatebox[2]{#2}%
  \newcommand*\fsize{\dimexpr\f@size pt\relax}%
  \newcommand*\lineheight[1]{\fontsize{\fsize}{#1\fsize}\selectfont}%
  \ifx\svgwidth\undefined%
    \setlength{\unitlength}{589.01869859bp}%
    \ifx\svgscale\undefined%
      \relax%
    \else%
      \setlength{\unitlength}{\unitlength * \real{\svgscale}}%
    \fi%
  \else%
    \setlength{\unitlength}{\svgwidth}%
  \fi%
  \global\let\svgwidth\undefined%
  \global\let\svgscale\undefined%
  \makeatother%
  \begin{picture}(1,0.266363)%
    \lineheight{1}%
    \setlength\tabcolsep{0pt}%
    \put(0,0){\includegraphics[width=\unitlength,page=1]{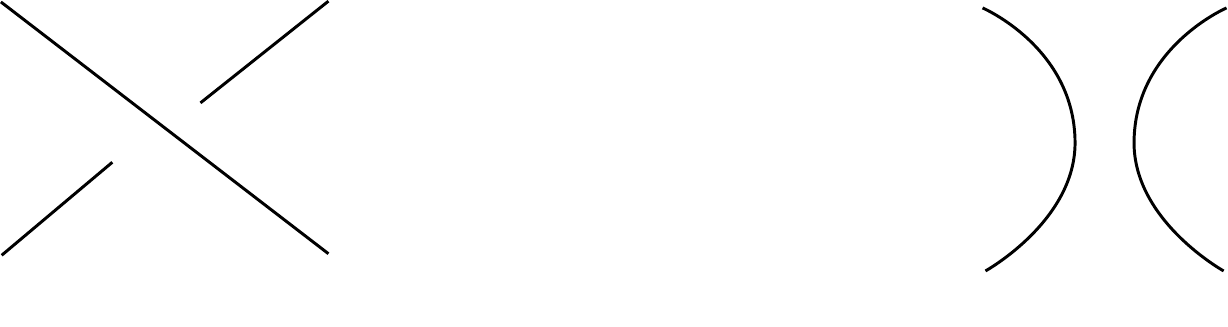}}%
    \put(0.40,-0.01519316){\color[rgb]{0,0,0}\makebox(0,0)[lt]{\lineheight{1.25}\smash{\begin{tabular}[t]{l}$0$-Smoothing\end{tabular}}}}%
    \put(0.8,-0.01526631){\color[rgb]{0,0,0}\makebox(0,0)[lt]{\lineheight{1.25}\smash{\begin{tabular}[t]{l}$1$-Smoothing\end{tabular}}}}%
    \put(0,0){\includegraphics[width=\unitlength,page=2]{smoothing.pdf}}%
  \end{picture}%
\endgroup%

    	}
    	\]
    	\caption{$0$- and $1$- smoothings}
    	\label{fig:smoothing}
    \end{figure}
    
    For each resolution $D_v$, define the graded vector space $\CKh(D_v)$ to be \[\CKh(D_v)\coloneqq V^{\otimes |D_v|}\left\{|v|\right\},\] where $|D_v|$ is the number of components in the smoothing $D_v$, and $\left\{|v|\right\}$ indicates the quantum grading shift, where \[|v| = \sum_{i=1}^nv_i.\]The \textit{Khovanov chain complex} $\CKh(D)$ is then defined to be 
    \[\CKh(D) \coloneqq \bigoplus_{v\in \underline{2}^n} CKh(D_v)[-n_-]\left\{n_+-2n_-\right\}, \]
    where $[-n_-]$ indicates the homological grading shift, and the differential is given as a sum over the edge maps in the cube, which is a tensor product of the identity map with either $m:V\otimes V \to V$ if it merges two circles into one, or $\Delta:V\to V\otimes V$ if it splits one circle into two. We denote the two gradings on $\CKh(D)$ by $(h,q)$, where $h$ is the homological grading and $q$ is the quantum grading. \textit{Khovanov homology} of $K$, denoted by $\Kh(K)$, is then defined as the homology of $\CKh(D)$. It is independent of the choice of link diagrams, and hence an invariant of the link.
   
   A cobordism $\Sigma$ between links $K_0$ and $K_1$ induces a chain map \[\CKh(\Sigma)\colon \CKh(K_0)\to \CKh(K_1)\]of bigrading $(0,\chi(\Sigma))$, defined by presenting the cobordism as a \textit{movie}, i.e., a sequence of planar isotopies, Reidemeister moves and Morse moves, and assigning a map for each elementary move. It is well-defined up to chain homotopy and is an invariant for cobordisms up to isotopy rel boundary. See for example \cite{bar2005khovanov} for the functoriality for cobordisms (up to a global sign) in $\mathbb{R}^3\times I$, and \cite{morrison2022invariants} for the functoriality in $S^3\times I$. 

   Most discussion above also applies to Bar-Natan homology \cite{bar2005khovanov}, where the Frobenius algebra $V$ is replaced by $\FF[H][X]/(X^2=HX)$, where $H$ is a formal variable of quantum degree $-2$. The output Bar-Natan chain complex $\CBN(D)$ is a finitely generated module over $\FF[H]$ whose homology is the \textit{Bar-Natan homology} $\BN(K)$.
    
\subsection{Tangle invariants}
\label{sec:tangle invariants background}The appropriate framework to compute Khovanov homology locally is the theory of tangle invariants developed in \cite{khovanov2002functor,bar2005khovanov}. We mostly use the formalism from \cite{khovanov2002functor}, although most of (if not all) the proofs also work under Bar-Natan's formalism with mild modifications.

An \textit{$(m,n)$-tangle diagram} $T$ is a vertical tangle diagram with $2m$ points on the top and $2n$ points on the bottom. Two tangle diagrams are equivalent if they are isotopic rel boundary. Compositions of tangles are performed vertically. More explicitly, if $a$ is an $(m,n)$-tangle and $b$ is a $(k,m)$-tangle, then $ba$ is the $(k,n)$-tangle obtained by identifying the bottom $2m$ endpoints of $b$ with the top $2m$ endpoints of $a$.

Let $\mathfrak{C}_n$ be the set of crossingless matchings between $2n$ points on a line, which can also be viewed as crossingless $(n,0)$-tangles. For $a\in \mathfrak{C}_n$, denote by $\bar{a}$ the reflection of $a$ about the line, which is a $(0,n)$-tangle. In \cite{khovanov2002functor}, Khovanov constructs a finite-dimensional graded ring 
\[H^n \coloneqq\bigoplus_{a,b\in \mathfrak{C}_n} {}_a (H^n)_b,\] with idempotents labeled by elements in $\mathfrak{C}_n$. Here, 
\[{}_a (H^n)_b \coloneqq V^{\otimes k_{a,b}}\left\{n \right\}\]
if the concatenation $\bar{b}a$ consists of $k_{a,b}$ many components, and $V$ is the same vector space assigned to the unknot in Khovanov homology. The nontrivial multiplication takes the form
\[{}_a (H^n)_b \otimes {}_b (H^n)_c \to {}_a (H^n)_c,\]
which is given by a sequence of saddle maps from $b\bar{b}$ to the identity tangle on $2n$ strands. For $n=0$, set $H^0$ to be the base field $\FF=\FF_2$.

For each $(m,n)$-tangle diagram $T$, Khovanov associates a chain complex of $(H^m,H^n)$-bimodules 
\[\mathcal{F}(T) \coloneqq \bigoplus_{a\in \mathfrak{C}_m} \bigoplus_{b\in \mathfrak{C}_n}{}_a (\mathcal{F}(T))_{b},\]
where ${}_a (\mathcal{F}(T))_{b} $ is the usual Khovanov chain complex associated to the concatenation $\bar{a}Tb$. In particular, $\mathcal{F}(K)$ coincides with the usual Khovanov chain complex for a knot $K$, viewed as a $(0,0)$-tangle. The assignment is independent of the choice of tangle diagrams up to chain homotopy equivalence.

We review some useful properties of the tangle invariants $\mathcal{F}(T)$ in the following lemmas. The first is a gluing result which states that the vertical concatenation of an $(m,n)$-tangle and a $(k,m)$-tangle corresponds to the tensor product of chain complexes of bimodules over the ring $H^m$. In particular,  this means that we can compute the Khovanov homology of links by cutting them into tangles, performing the local computation for each tangle, and then gluing the tangles back. 

\begin{lem} [\cite{khovanov2002functor}*{Theorem~1}]\label{lem:concatenation of tangles}
    For an $(m,n)$-tangle $a$ and a $(k,m)$-tangle $b$, there is a canonical isomorphism of $(H^k,H^n)$-bimodules
    \[\mathcal{F}(ba) \cong \mathcal{F}(b)\otimes_{H^m}\mathcal{F}(a).\]
\end{lem}

The next one concerns the exactness of the tensor product, allowing us to replace the tangle invariants and maps between them with homotopic ones in the computation. 
\begin{lem} [\cite{khovanov2002functor}*{Proposition~3}]\label{lem:sweet bimodules}
    For an $(m,n)$-tangle $a$, the tensor product functor:
    \[-\otimes_{H^m} \mathcal{F}(a) \colon M \to M\otimes_{H^m} \mathcal{F}(a)  \] is an exact functor from the category of right $H^m$-modules to the category of right $H^n$-modules, and a similar statement holds for $\mathcal{F}(a) \otimes_{H^n}-$.
\end{lem}

Khovanov uses the terminology \textit{sweet bimodule} \cite{khovanov2002functor}*{Definition~1} for $(H^m,H^n)$-bimodules that are finitely-generated and projective as both left $H^m$-modules and right $H^n$-modules and proves that $\mathcal{F}(a)$ is a sweet bimodule for any $(m,n)$-tangle $a$.

The next lemma states the functoriality of the tangle invariants under tangle cobordisms.

\begin{lem} [\cite{khovanov2006invariant}*{Theorem~1}] \label{lem:functoriality of tangle invariants}
    For each relative isotopy class of tangle cobordism $\Sigma:T_1 \to T_2$ (in $\mathbb{R}^2 \times [0,1]\times [0,1]$), there is a cobordism map
    \[\mathcal{F}_{\Sigma} \colon \mathcal{F}(T_1) \to \mathcal{F}(T_2),\]well-defined up to chain homotopy. 
\end{lem}

In fact, we get a $2$-functor $\mathcal{F}: \mathcal{T} \to \hat{\mathbb{K}}.$  Here, $\mathcal{T}$ is the $2$-category with objects consisting of even-length strings of signs, $1$-morphisms of (oriented) tangles between endpoints with compatible orientations, and $2$-morphisms of (oriented) tangle cobordisms. $\hat{\mathbb{K}}$ is the $2$-category with objects of non-negative integers, $1$-morphisms of bounded chain complexes of $(H^m,H^n)$-bimodules, and $2$-morphisms of chain maps up to homotopy. This is an overall sign ambiguity in Khovanov's formalism of the $2$-functor, but it will not be an issue when working over $\FF_2$. 

One of the key steps in establishing functoriality, and of independent importance to us, is a rigidity property for a certain class of tangles, which we now recall. See also  \cite[Definition 8.5]{bar2005khovanov} and the discussion around it.

\begin{lem}[\cite{khovanov2006invariant}*{Corollary~1,2}] \label{lem:rigidity of tangle invariants}
    Let $a$ be an $(n,n)$-tangle such that there exist $(n,n)$-tangles $b,c$ that the concatenations $ba$ and $ac$ are both the identity $(n,n)$-tangle. Then, every degree $0$ homotopy equivalence $f\colon \mathcal{F}(a) \to \mathcal{F}(a)$ is chain homotopic to the identity map $id_{\mathcal{F}(a)}$. Similarly, if $a$ and $a'$ are two tangles satisfying the above conditions, then all homotopy equivalences between $\mathcal{F}(a)$ and $\mathcal{F}(a')$ are chain homotopic. 
\end{lem}

The conditions on the concatenations of the tangle $a$ imply that the tangle invariant $\mathcal{F}(a)$ is \textit{invertible} in Khovanov's language in \cite{khovanov2006invariant}, and \textit{Kh-simple} in Bar-Natan's language in \cite{bar2005khovanov}. In particular, this holds if $a$ is a $2n$-strand braid. Again, there is no sign ambiguity when working over $\FF_2$. See also \cite{lipshitz2018bordered}*{Section~4} for the corresponding statements in bordered Floer homology.

\section{Involutive Khovanov homology}\label{sec:KhI}

In this section, we use the flip map to define the involutive Khovanov homology for links in $S^3$, prove that it is a link invariant, and compute it for unlinks. 

We first recall some definitions from the introduction.

\begin{defn}
	Let $K\subset S^3$ be a link and $D$ be a diagram for $K$.  The \textit{flip diagram} for $D$ is the diagram $D^*$ obtained by taking the mirror diagram of $D$ across a line and then reversing all crossings (i.e., switching the undercrossings with overcrossings). Equivalently, it is obtained by rotating $K$ about an axis in the plane and then projecting it onto the same plane.
\end{defn}

It is clear from the definition that $D^*$ represents the same link $K$, and it is independent of the choice of the axis up to planar isotopy.

There is a canonical identification between resolutions of $D$ and $D^*$ by  flipping all circles in the resolutions, which gives a canonical identification between the Khovanov chain complexes \[\eta_D\colon \CKh(D)\to\CKh(D^*).\]The diagrams $D$ and $D^*$ are related by a preferred isotopy, namely rotating $D$ about the axis in $S^3$ by $\pi$. We call the trace cobordism of this isotopy the \textit{flip cobordism}. This gives another chain homotopy equivalence \[R_{D^*}\colon \CKh(D^*)\to\CKh(D),\]induced by a sequence of Reidemeister moves that realizes the flip cobordism, which is well-defined up to chain homotopy by functoriality.

\begin{rem}
    There are \textit{a priori} two choices of the flip cobordism: the trace of rotating $D$ by $\pi t\,(t\in[0,1])$, or the trace of rotating $D$ by $-\pi t \,(t\in[0,1])$. The difference between the two choices is the cobordism map of the trace of a $2\pi$ rotation of $K$, which is proved in Lemma \ref{lem: rotating by 2pi} to be the identity map. Therefore, two choices yield the same chain map up to homotopy.
\end{rem}

\begin{defn}
	Let $D$ be a diagram of $K$. The \textit{flip map} $\kappa_D$ is defined as the composition $R_{D^*}\circ\eta_D$: \[\kappa_D\colon\CKh(D)\xrightarrow{\eta_D}\CKh(D^*)\xrightarrow{R_{D^*}}\CKh(D).\]
\end{defn}

We often suppress the subscript in $\kappa_D$ and simply denote the flip map by $\kappa$ when the diagram $D$ we consider is clear from the context. To prove $\kappa$ is a homotopy involution, we need the following lemma, which seems to be well-known to experts.

\begin{lem}\label{lem: rotating by 2pi}
    Let $D$ be a diagram for a knot $K\subset S^3$. Let $\Sigma$ be the trace of the isotopy that rotates $K$ about an arbitrary axis by $2\pi$. Then $\CKh(\Sigma)\colon \CKh(D)\to\CKh(D)$ is chain homotopic to the identity map.
\end{lem}

\begin{proof}
   If the projection plane for $D$ is orthogonal to the rotational axis, then $\Sigma$ can be realized as a planar isotopy, and hence the assertion trivially holds. In general, let $D'$ be a diagram for a knot isotopic to $K$ such that the corresponding projection plane is orthogonal to the rotational axis, and $\Sigma'$ be the corresponding trace of rotation for $D'$. Let $\varphi\colon\CKh(D)\to\CKh(D')$ be a chain homotopy equivalence induced by any fixed sequence of Reidemeister moves from $D$ to $D'$. The ``rotation by $2\pi$'' cobordism commutes with any cobordism represented by Reidemeister moves up to isotopy relative to the boundary, and hence, the following diagram commutes up to chain homotopy, by the naturality of Khovanov homology:
    \begin{equation*}
    \begin{tikzcd}
     \CKh(D)\ar[r,"\CKh(\Sigma)"] \ar[d,"\varphi"]& \CKh(D)\ar[d,"\varphi"] \\
     \CKh(D') \ar[r,"\CKh(\Sigma')"]& \CKh(D')
\end{tikzcd}.
\end{equation*}Here, the bottom line is the identity map. Therefore we have \[\CKh(\Sigma)\simeq\varphi^{-1}\circ\CKh(\Sigma')\circ\varphi\simeq\idd.\]
\end{proof}

\begin{lem}\label{lem: kappa hmtp involution}
    The flip map $\kappa$ is a homotopy involution, \textit{i.e.,} $\kappa^2 \simeq \idd_{\CKh(D)}$.
\end{lem}

\begin{proof}
  Note that 
    \[\eta_D \circ R_{D^*} \circ \eta_D \simeq R_{D},\]
    where $R_D: \CKh(D) \to \CKh(D^*)$ is induced by a sequence of Reidemeister moves rotating the diagram $D$ by $\pi t\, (t\in [0,1])$ about the axis to the diagram $D^*$.
 Then,
    \[\kappa_D^2 = \eta_D \circ R_{D^*} \circ \eta_D \circ R_{D^*} \simeq R_D\circ R_{D^*} \simeq id_{\CKh(D)}.\] Here, $R_D\circ R_{D^*}$ is the trace cobordism of rotating $D$ about the axis by $2\pi$. The claim then follows from Lemma \ref{lem: rotating by 2pi}.
\end{proof}

As an analog of involutive Heegaard Floer homology, we define the involutive Khovanov homology as follows.

\begin{defn}
    The \textit{involutive Khovanov chain complex}, denoted by $\CKhI(D)$, is a bigraded $\FF[Q]/Q^2$-module defined by \[\CKhI(D)\coloneqq \operatorname{Cone}(\CKh(D)\xrightarrow{Q\cdot(\idd-\kappa_D)}Q\cdot \CKh(D)[1]),\]where $D$ is a diagram of $K$, and $Q$ is a formal variable of bigrading $(1,0)$.
    The \textit{involutive Khovanov homology} of $K$, denoted by $\KhI(K)$, is the (co)homology of $\CKhI(D)$.
\end{defn}
\begin{warn}
    The reader should not confuse our version of involutive Khovanov homology, defined utilizing the \textit{intrinsic} symmetry, with the involutive Khovanov homology for strongly invertible knots defined by Sano \cite{sano2024involutive}, which is defined using strong inversions as the \textit{external symmetry}.
\end{warn}

The invariance of involutive Khovanov homology is a formal adaptation of the approach in \cite{hendricks2017involutive}. The only difference is that while a Heegaard diagram is associated to a $3$-manifold, a knot diagram is associated to an \textit{isotopy class of knots}, so we need to be slightly more careful when proving the invariance.

\begin{thm} \label{prop:well-definedness of KhI}
 Let $K\subset S^3$ be a link. The quasi-isomorphism class of the involutive Khovanov chain complex $\CKhI(D)$ is independent of the choice of link diagrams, and hence, it is an invariant of $K$.
\end{thm}

\begin{proof}
    This follows from the same argument for the well-definedness of the involutive Heegaard Floer homology as in \cite{hendricks2017involutive}*{Proposition~2.7}. Suppose $D_1$ and $D_2$ are two diagrams of the same link $K$. Let \[\varphi\colon \CKh(D_1) \to \CKh(D_2)\] be the map induced by a sequence of Reidemeister moves from $D_1$ to $D_2$, and 
    \[\varphi^*\colon \CKh(D_1^*) \to \CKh(D_2^*)\]be the map induced by the sequence of Reidemeister moves obtained from the previous one by flipping each step. Then, from the definition of $\eta$, we have 
    \[\varphi^* \circ \eta_{D_1} = \eta_{D_2} \circ \varphi.\]

    From the definition of $\varphi^*$, it is clear that $R_{D_2^*}\circ \varphi^*$ and $\varphi \circ R_{D_1^*}$ are two maps induced by two sequences of Reidemeister moves from $D_1^*$ to $D_2$, and two cobordisms are isotopic relative to the boundary. These two maps are homotopic by functoriality. Therefore, we have 
    \[ \varphi \circ \kappa_{D_1}= \varphi \circ R_{D_1^*}\circ \eta_{D_1}\simeq R_{D_2^*} \circ \varphi^* \circ \eta_{D_1} = R_{D_2^*} \circ \eta_{D_2} \circ \varphi = \kappa_{D_2} \circ \varphi, \]
    and \[ \varphi \circ Q(\idd_{\CKh(D_1)} - \kappa_{D_1}) \simeq Q(\idd_{\CKh(D_2)}-\kappa_{D_2})\circ \varphi.\]Denote the homotopy by $\gamma\colon\CKh(D_1) \to Q\cdot\CKh(D_2)[1]$. Then the map
    \[\begin{pmatrix}
        \varphi & 0\\
        \gamma & \varphi
    \end{pmatrix}: \CKhI(D_1) \to \CKhI(D_2)\]
    yields the desired quasi-isomorphism by standard homological algebra.
\end{proof}

We now explicitly compute the flip map for the unknot and unlinks, which will be useful in the proof of Theorem \ref{thm kappa=id}. It turns out that the flip map is the identity map \textit{on the nose} in these cases.  As we will see in Section \ref{sec: discussions}, this only holds for the ordinary Khovanov homology with $\FF= \FF_2$ coefficients.

\begin{exmp}[the flip map for the unknot]
\label{exmp unknot kappa}

Let $D$ be the diagram of the unknot $U$ with no crossings. Then $D^*$ is the same as $D$ up to planar isotopy, and $\eta_D$ is the identity map $\idd_V$ on $\CKh(D^*) = \CKh(D) = V$.

    \begin{figure}[h!]
	\[{
		\fontsize{7pt}{9pt}\selectfont
		\def\svgscale{0.8}
		%% Creator: Inkscape 1.2.1 (9c6d41e410, 2022-07-14), www.inkscape.org
%% PDF/EPS/PS + LaTeX output extension by Johan Engelen, 2010
%% Accompanies image file 'flip_on_unknot.pdf' (pdf, eps, ps)
%%
%% To include the image in your LaTeX document, write
%%   \input{<filename>.pdf_tex}
%%  instead of
%%   \includegraphics{<filename>.pdf}
%% To scale the image, write
%%   \def\svgwidth{<desired width>}
%%   \input{<filename>.pdf_tex}
%%  instead of
%%   \includegraphics[width=<desired width>]{<filename>.pdf}
%%
%% Images with a different path to the parent latex file can
%% be accessed with the `import' package (which may need to be
%% installed) using
%%   \usepackage{import}
%% in the preamble, and then including the image with
%%   \import{<path to file>}{<filename>.pdf_tex}
%% Alternatively, one can specify
%%   \graphicspath{{<path to file>/}}
%% 
%% For more information, please see info/svg-inkscape on CTAN:
%%   http://tug.ctan.org/tex-archive/info/svg-inkscape
%%
\begingroup%
  \makeatletter%
  \providecommand\color[2][]{%
    \errmessage{(Inkscape) Color is used for the text in Inkscape, but the package 'color.sty' is not loaded}%
    \renewcommand\color[2][]{}%
  }%
  \providecommand\transparent[1]{%
    \errmessage{(Inkscape) Transparency is used (non-zero) for the text in Inkscape, but the package 'transparent.sty' is not loaded}%
    \renewcommand\transparent[1]{}%
  }%
  \providecommand\rotatebox[2]{#2}%
  \newcommand*\fsize{\dimexpr\f@size pt\relax}%
  \newcommand*\lineheight[1]{\fontsize{\fsize}{#1\fsize}\selectfont}%
  \ifx\svgwidth\undefined%
    \setlength{\unitlength}{412.73733035bp}%
    \ifx\svgscale\undefined%
      \relax%
    \else%
      \setlength{\unitlength}{\unitlength * \real{\svgscale}}%
    \fi%
  \else%
    \setlength{\unitlength}{\svgwidth}%
  \fi%
  \global\let\svgwidth\undefined%
  \global\let\svgscale\undefined%
  \makeatother%
  \begin{picture}(1,0.21266044)%
     \lineheight{1}%
 \setlength\tabcolsep{0pt}%
 \put(0,0){\includegraphics[width=\unitlength,page=1]{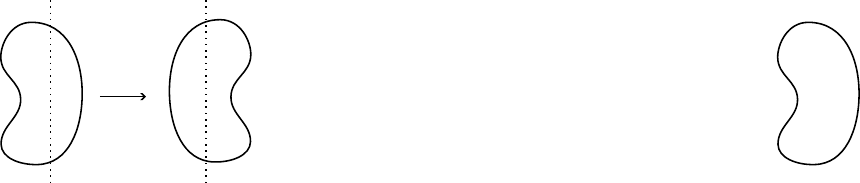}}%
 \put(0.13536988,0.11197517){\color[rgb]{0,0,0}\makebox(0,0)[lt]{\lineheight{1.25}\smash{\begin{tabular}[t]{l}$\eta$\end{tabular}}}}%
 \put(0,0){\includegraphics[width=\unitlength,page=2]{flip_on_unknot.pdf}}%
 \put(0.31,0.11197517){\color[rgb]{0,0,0}\makebox(0,0)[lt]{\lineheight{1.25}\smash{\begin{tabular}[t]{l}An $R$-II move\end{tabular}}}}%
 \put(0,0){\includegraphics[width=\unitlength,page=3]{flip_on_unknot.pdf}}%
 \put(0.52,0.11197517){\color[rgb]{0,0,0}\makebox(0,0)[lt]{\lineheight{1.25}\smash{\begin{tabular}[t]{l}An $R$-I move\end{tabular}}}}%
 \put(0,0){\includegraphics[width=\unitlength,page=4]{flip_on_unknot.pdf}}%
 \put(0.73,0.11197517){\color[rgb]{0,0,0}\makebox(0,0)[lt]{\lineheight{1.25}\smash{\begin{tabular}[t]{l}Another $R$-I move\end{tabular}}}}%
 \put(0.73,0.07526124){\color[rgb]{0,0,0}\makebox(0,0)[lt]{\lineheight{1.25}\smash{\begin{tabular}[t]{l} and planar isotopy\end{tabular}}}}%
  \end{picture}%
\endgroup%

	}\]
	\caption{The flip map for the unknot $U$}
    \label{fig unknot kappa}
\end{figure}
    
    One way to realize the trace cobordism map $R_{D^*}\colon\CKh(D^*) \to \CKh(D)$ is depicted in Figure \ref{fig unknot kappa}, where we first perform a Reidemeister II move and then perform two Reidemeister I moves. A direct computation shows that 
    \[R_{D^*}\colon \CKh(D^*) \to \CKh(D) \]
   sends $1$ to $1$ and $X$ to $X$, so $\kappa_D\colon \CKh(D) \to \CKh(D) $ is the identity map, and 
    \[\KhI(U) \cong \Kh(U)\otimes_\FF\FF[Q]/Q^2.\]
    
    Note that there is a unique bigrading-preserving isomorphism from $\CKh(D^*)$ to $\CKh(D)$ with $\mathbb{F}$ coefficients, so the calculation is in fact automatic. We draw the diagram to indicate how we will realize the trace cobordism map $R_{D^*}$ in the general case when $D$ is a plat closure of a braid.
 
\end{exmp}

\begin{exmp}[unlinks with crossingless diagrams]\label{exmp unlink kappa}

    We first consider a simple case, where the unlink is displayed as in Figure \ref{fig:flip on unlink}. Let $D$ be this diagram of the $n$-component unlink $U_n$ with no crossings. In this position, different components of $D$ do not interact during the flipping process. Similar to the argument for the unknot, 
    \[\kappa_D\colon\CKh(D)\to \CKh(D)\]
    is the identity map, and 
    \[\KhI(U_n) \cong \Kh(U_n)\otimes_\FF\FF[Q]/Q^2.\]

    \begin{figure}[ht]
    	\[{
    		\fontsize{7pt}{9pt}\selectfont
    		\def\svgscale{0.8}
    		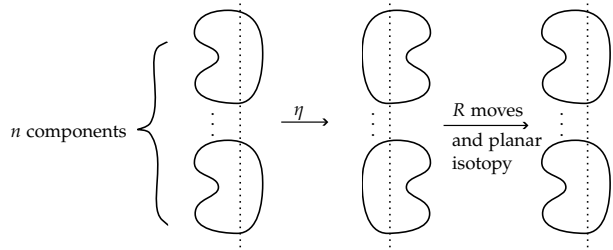
    	}\]
    	\caption{The flip map for the unlink $U_n$ with a particular crossingless diagram $D$}
    	\label{fig:flip on unlink}
    \end{figure}

    For later purposes, we also need to compute the flip map $\kappa_{\mathfrak{D}}$ for a diagram $\mathfrak{D}$ with no crossings, but with an arbitrary relative position with the rotational axis. The trace cobordism map $R_{\mathfrak{D}}\colon\CKh(\mathfrak{D}^*) \to \CKh(\mathfrak{D})$ could \textit{a priori} be complicated. However, the flip map $\kappa_{\mathfrak{D}}$ turns out to be always equal to the identity map by the following argument. Consider the following commutative diagram, constructed using the naturality of Khovanov homology as in the proof of Theorem \ref{prop:well-definedness of KhI}:
    \[
    \begin{tikzcd}[row sep = 10mm, column sep =  10mm]
        \CKh(D) \ar[r,"\kappa_D"] \ar[d,"\varphi"] \ar[dr, dashed, "H"]& \CKh(D) \ar[d,"\varphi"]\\
        \CKh(\mathfrak{D}) \ar[r,"\kappa_{\mathfrak{D}}"] & \CKh(\mathfrak{D})
    \end{tikzcd}.
    \]
    Here $\varphi$ is induced by a fixed sequence of Reidemeister moves from $D$ to $\mathfrak{D}$, and $H$ is a homotopy between $\varphi \circ \kappa_D$ and $\kappa_{\mathfrak{D}} \circ \varphi$. However, since $D$ and $\mathfrak{D}$ are both crossingless, each of the chain complex $\CKh(D)$ and $\CKh(\mathfrak{D})$ is concentrated in homological grading $0$. Therefore, we have 
   $H=0$, and the diagram above strictly commutes. Similarly, $\varphi\colon\CKh(D) \to \CKh(\mathfrak{D})$ is \textit{a priori} only a homotopy equivalence, but as both $D$ and $\mathfrak{D}$ are crossingless, $\varphi$ is an isomorphism of Khovanov chain complexes. Denote its inverse by $\varphi^{-1}$. Then we have
    \[\kappa_{\mathfrak{D}} = \varphi \circ \kappa_{D} \circ \varphi^{-1} = \varphi \circ \idd_{\CKh(D)} \circ\varphi^{-1} = \idd_{\CKh(\mathfrak{D})}.\]
    
\end{exmp}

\section{The flip map}
\label{sec:flip map}
In this section, we prove Theorem \ref{thm kappa=id} that the flip map $\kappa_D$ is chain homotopic to the identity map. The main difficulty is to compute the map $R_{D^*}\colon\CKh(D^*)\to \CKh(D)$ induced by Reidemeister moves that realizes the cobordism map. 

We begin by identifying a sequence of Reidemeister moves that is particularly easy to understand locally. In fact, as we claimed in Theorem \ref{thm kappa=id}, $R^{-1}_{D^*}$ turns out to be chain homotopic to the algebraic identification $\eta_{D}$. The strategy here is similar to the proof of \cite[Theorem 4.1]{alishahi2023khovanov}; we will discuss the relation with that work in more detail in Section \ref{sec: discussions}.

We first introduce some elementary topological facts regarding the flip cobordism on tangles before explicitly computing the flip map. Recall from the convention in Section \ref{sec:tangle invariants background} that an $(m,n)$-tangle has $2m$ endpoints on the top and $2n$ endpoints on the bottom. Denote the positive half twist on $2k$ strands by $\Delta_{2k}$, presented as the braid word
\[\Delta_{2k}=(\sigma_{2k-1}\sigma_{2k-2}\cdots\sigma_1)(\sigma_{2k-1}\cdots\sigma_{2})\cdots(\sigma_{2k-1}\sigma_{2k-2})\sigma_{2k-1}.\]

\begin{lem}\label{lem flip tangles}
	Let $T$ be a $(k,k)$-tangle represented by a single braid word $\sigma_i$ or $\sigma^{-1}_i$. Then $\Delta_{2k}^{-1}\circ T\circ \Delta_{2k}$ is isotopic to $T^{*}= \sigma_{2k-i}$ or $\sigma_{2k-i}^{-1}$ respectively. Moreover, they are related by a cobordism between tangles that rotates the tangle about the $y$-axis by $\pi$. See the left of Figure \ref{fig:flip_one_crossing} for illustration.
\end{lem}

\begin{lem}\label{lem flip k cups}
	Let $T$ be the $(0,k)$-tangle with $k$ caps, connecting the first and second strands, ..., the $(2k-1)$-th and the $(2k)$-th strands. Then $T\circ\Delta_{2k}$ is isotopic to $T^*(=T)$. Similarly, if $T$ is the $(k,0)$-tangle that consists of $k$ cups, then $\Delta_{2k}^{-1}\circ T$ is isotopic to $T^*(=T)$. Moreover, they are related by a cobordism between tangles that rotates the tangle about the $y$-axis  by $\pi$. See the right of Figure \ref{fig:flip_one_crossing} for illustration.
\end{lem}
\begin{figure}[ht]
	\[{
		\fontsize{7pt}{9pt}\selectfont
		\def\svgscale{1}
		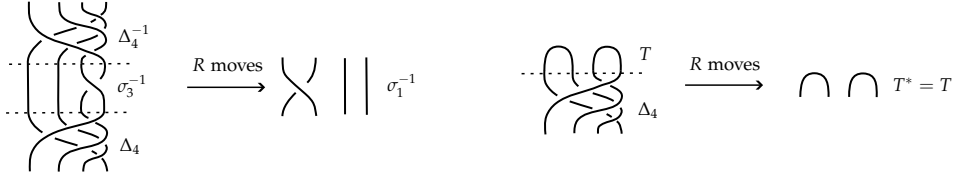
	}\]
	\caption{Diagrammatic illustrations of Lemmas \ref{lem flip tangles} and \ref{lem flip k cups}. Left: apply a pair of half twists to flip a simple braid. Right: apply a half twist to flip the caps.}
	\label{fig:flip_one_crossing}
\end{figure}
The proofs are by performing the obvious Reidemeister moves; see Figure \ref{fig:flip_one_crossing}.

 We can now explicitly write down the flip map. The idea of the construction is that the half twist can serve as the ``universal object'' that realizes the flip cobordism. 

\begin{con} \label{con: flip map}Let $K\subset S^3$ be a link, and $D$ be a diagram for $K$ which is the plat closure of a braid with $2k$ strands and $n$ crossings, presented vertically. We construct a sequence of Reidemeister moves from $D^*$ to $D$, as follows.

    \begin{enumerate}[i.]
    \item Decompose $D$ as a composition of $n+2$ pieces: \[D=T_0\circ T_1\circ\cdots\circ T_n\circ T_{n+1},\] where $T_0$ and $T_{n+1}$ contain $k$ caps and cups respectively, and each $T_i$ is a tangle with a single crossing for $i=1,2,\dots,n$. Then $D^*=T^*_0\circ T^*_1\circ\cdots\circ T^*_n\circ T^*_{n+1}$, where $T^*_i$ is the flip of $T_i$ across the $y$-axis.
    
    \item Introduce a pair of canceling half twists $\Delta_{2k} \circ \Delta_{2k}^{-1}$ between $T^*_i$ and $T^*_{i+1}$ ($0\le i\le n$) via a sequence of Reidemeister II moves. 
    
    \item Isotope $\Delta^{-1}_{2k} \circ T^*_i \circ \Delta_{2k}$ to $T_i$ for $1\le i\le n$, using Reidemeister moves described in Lemma \ref{lem flip tangles}.
    
    \item Isotope $T_0^*\circ\Delta_{2k}$ to $T_0$  and $\Delta_{2k}^{-1}\circ T^*_{n+1} $ to $T_{n+1}$, using Reidemeister I and II moves described in Lemma \ref{lem flip k cups}.
\end{enumerate}
\end{con}

See Figure \ref{fig:main contruction} below for an example.

 \begin{figure}[ht]
	\[{
		\fontsize{7pt}{9pt}\selectfont
		\def\svgscale{1}
		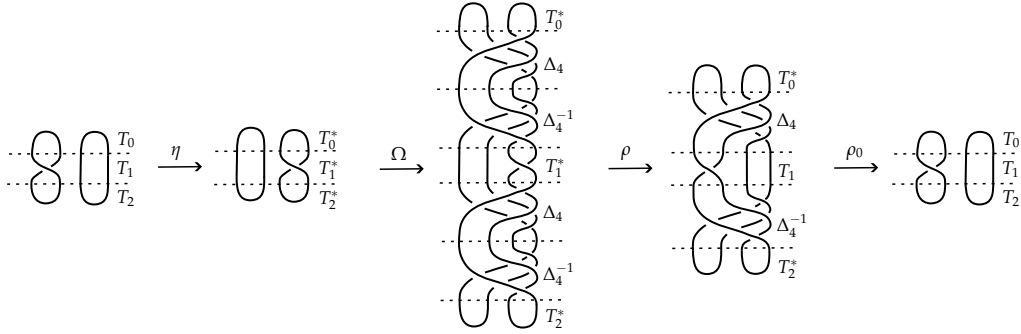
	}\]
	\caption{An example of the flip map for $D = T_0\circ T_1\circ T_2$, where we apply $\mathcal{F}$ to each of the diagram and apply the corresponding maps in Equation (\ref{eqn decompose kappa}).}
	\label{fig:main contruction}
\end{figure}

By Lemma \ref{lem:concatenation of tangles}, we have canonical identifications of chain complexes
\begin{align}
    \label{eqn:cutting into tangles}
    \begin{split}
         \CKh(D)&\cong \caf(T_0)\otimes\caf(T_1)\otimes\cdots\otimes \caf(T_n)\otimes\caf(T_{n+1}),\\
     \CKh(D^*)&\cong \caf(T_0^*)\otimes\caf(T_1^*)\otimes\cdots\otimes \caf(T_n^*)\otimes\caf(T_{n+1}^*),
    \end{split}
\end{align}
where the tensor products are taken over the algebra $H^k$.

To compute the flip map $\kappa$, we note that Construction \ref{con: flip map} induces a sequence of maps on tangle invariants:
\begin{align}\label{eqn decompose kappa}
	\begin{split}
		\CKh(D)&\xrightarrow{\eta} \CKh(D^*)\\
		&\cong\caf(T_0^{*})\otimes\caf(T^*_1)\otimes\cdots\otimes \caf(T^*_n)\otimes\caf(T^*_{n+1})\\
		&=\caf(T^*_0)\otimes I_k\otimes \caf(T^*_1)\otimes\cdots\otimes I_k\otimes\caf(T^*_{n+1})\\
		&\xrightarrow{\Omega}\caf(T^*_0)\otimes\left(\caf(\Delta_{2k})\otimes\caf(\Delta_{2k}^{-1})\right)\otimes \caf(T^*_1)\otimes\cdots\otimes\caf(T^*_{n+1})\\
		&=\left(\caf(T^*_0)\otimes\caf(\Delta_{2k})\right)\otimes\left(\caf(\Delta_{2k}^{-1})\otimes\caf(T^*_1)\otimes\caf(\Delta_{2k})\right)\otimes\dots\otimes\left(\caf(\Delta_{2k}^{-1})\otimes\caf(T^*_{n+1})\right)\\
		&\xrightarrow{\rho}\left(\caf(T^*_0)\otimes\caf(\Delta_{2k})\right)\otimes\caf(T_1)\otimes\cdots\otimes\left(\caf(\Delta_{2k}^{-1})\otimes\caf(T^*_{n+1})\right)\\
		&\xrightarrow{\rho_0}\caf(T_0)\otimes\caf(T_1)\otimes\cdots\otimes\caf(T_{n+1})\\
		&\cong\CKh(D).
	\end{split}
\end{align}
We explain the notations appeared in (\ref{eqn decompose kappa}).
\begin{itemize}
    \item $\mathcal{F}(T)$ is Khovanov's tangle invariant \cite{khovanov2002functor}, reviewed in Section \ref{sec:tangle invariants background}.

    \item All tensor products are taken over the algebra $H^{k}$.

    \item $I_k$ is the identity $(H^k,H^k)$-bimodule associated to the trivial tangle with $2k$ strands.

    \item $\Omega$, $\rho$, and $\rho_0$ are maps from steps ii, iii, and iv of Construction \ref{con: flip map} respectively.

    \item The second and the last row are identifications in Equation (\ref{eqn:cutting into tangles}). 
\end{itemize}

\begin{prop}
    The map described in Equation (\ref{eqn decompose kappa}) is chain homotopic to the flip map $\kappa_D$.
\end{prop}

\begin{proof}
    This follows from Lemmas \ref{lem flip tangles} and \ref{lem flip k cups}, and naturality of Khovanov homology.
\end{proof}

The main technical ingredient for computing the map $\rho$ (up to homotopy) is the following lemma, which is inspired by the cup-sliding arguments in \cite{rozansky2010categorification,willis2021khovanov}. Roughly speaking, we replace the map $\rho$ in (\ref{eqn decompose kappa}) by a homotopic one such that the resulting map is formed by the flip map on resolutions plus some homotopy.
\begin{lem}\label{lem algebraic flipping of an elementary tangle}
    Let $T$ be a $(k,k)$-tangle representing a single braid word $\sigma_i$ (resp. $\sigma^{-1}_i$) and $T^*$ be its flipped tangle representing $\sigma_{2k-i}$ (resp. $\sigma^{-1}_{2k-i}$). There exists a map \[\varphi\colon \caf(\Delta_{2k}^{-1}\circ T^{*} \circ \Delta_{2k}) \to \caf(T)\]with the following properties:

    \begin{enumerate}[i.]
    \item $\varphi$ is chain homotopic to the map induced by a sequence of Reidemeister moves realizing the flip cobordism.
    \item $\varphi$ is filtered with respect to the $\left\{0,1\right\}$-filtration arising from resolving the unique crossings in $T$ and $T^*$.
    \item The filtered pieces of $\varphi$ (with respect to the aforementioned filtrations) are equal to the maps induced by sequences of Reidemeister moves from $\caf(\Delta_{2k}^{-1}\circ T^{*}_{j} \circ \Delta_{2k}) \to \caf(T_j)$ ($j=0,1$) realizing the flip cobordism, where $T_j$ is the $j$-th resolution of $T$.  
    \end{enumerate}
\end{lem}
See Figure \ref{fig:algebraic replacement } for a schematic illustration of $\varphi$.

\begin{figure}[ht]
	\[{
		\fontsize{7pt}{9pt}\selectfont
		\def\svgscale{1}
		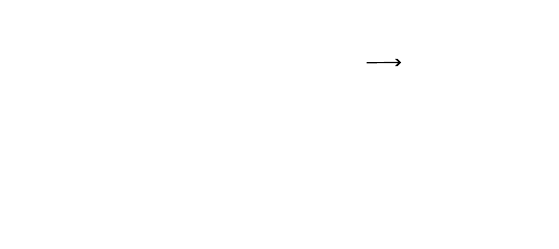
	}\]
	\caption{Schematic drawing of the map $\varphi\colon \caf(\Delta_{2k}^{-1}\circ T^{*} \circ \Delta_{2k}) \to \caf(T)$.}
	\label{fig:algebraic replacement }
\end{figure}

\begin{proof}
We prove the statement for $T=\sigma_i^{-1}$, and the other case that $T=\sigma_i$ follows the same proof by switching the $0$-and $1$-resolutions.

We first introduce some notations. Let $T_j$ and $T^*_j$ be the $j$-resolution ($j=0,1$) of $T$ and $T^*$ respectively. Let $f\colon\caf(T^*_0) \to \caf (T^*_1)$ be the map induced by the saddle cobordism between the $0$- and $1$-resolutions of $T^*$. By abuse of notation, we also denote the map
\[\idd\otimes f\otimes \idd\colon\caf(\Delta^{-1}_{2k}\circ T^*_0\circ \Delta_{2k}) \to \caf (\Delta^{-1}_{2k}\circ T^*_1\circ \Delta_{2k})\] by $f$. Then we have
\[\caf(\Delta^{-1}_{2k} \circ T^* \circ \Delta_{2k}) =\Cone\big(\caf(\Delta^{-1}_{2k}\circ T^*_0\circ \Delta_{2k}) \xrightarrow{f} \caf(\Delta^{-1}_{2k}\circ T^*_1\circ \Delta_{2k})\big).\]

Let $g\colon\caf(T_0) \to \caf(T_1)$ be the map induced by the saddle cobordism. Similarly, we have \[\caf(T) = \Cone\left( \caf(T_0)\xrightarrow{g}\caf(T_1)\right).\]
Let \[R_j\colon\caf(\Delta^{-1}_{2k}\circ T^*_j\circ \Delta_{2k}) \to \caf(T_j)\,(j=0,1)\] denotethe maps induced by sequences of Reidemeister moves that realize the flip cobordisms on $\Delta^{-1}_{2k}\circ T^*_j\circ \Delta_{2k}$ ($j=0,1$). Since the two cobordisms are isotopic relative to the boundary, by functoriality we have \[ g\circ R_0 \simeq R_1 \circ f.\] 
Hence, 
\[g \simeq R_1 \circ f \circ R_0^{-1},\]
where $R_0^{-1}$ is obtained by undoing the Reidemeister moves inducing $R_0$.

Note also that $R_0$ only involves Reidemeister moves I and II which decrease the number of crossings, and $R_1$ only involves Reidemeister moves II which decrease the number of crossings. Therefore, $R_j$ ($j=0,1$) are \textit{very strong deformation retracts} in the sense of \cite[Definition 2.9]{willis2021khovanov} by \cite[Lemma 2.13]{willis2021khovanov}. Hence, one can use the standard homological perturbation lemma (see, for example, \cite[Proposition 2.10]{willis2021khovanov}) to get another mapping cone 
\[\mathcal{C}(T)\coloneqq \Cone\left( \caf(T_0) \xrightarrow{R_1 \circ f \circ R_0^{-1}}  \caf(T_1)\right)\]
with an explicit chain homotopy equivalence \[H\colon\caf(\Delta^{-1}_{2k}\circ T^*\circ \Delta_{2k}) \to \mathcal{C}(T)\]  of the following form:
\begin{equation}
    \begin{tikzcd}
    \caf(\Delta^{-1}_{2k}\circ T^*\circ \Delta_{2k})\ar[d,"H"]&=& \caf(\Delta^{-1}_{2k}\circ T^*_0\circ \Delta_{2k}) [-1]\ar[r,"f"] \ar[d,"R_0"] \ar[dr, dashed, "h"']& \caf(\Delta^{-1}_{2k}\circ T^*_1\circ \Delta_{2k})\ar[d,"R_1"]\\
    \mathcal{C}(T)&=& \caf(T_0)[-1] \ar[r,"R_1\circ f\circ R^{-1}_0"']& \caf(T_1)
\end{tikzcd},
\end{equation}
where the dashed map $h$ is given by the homological perturbation lemma\footnote{We do not spell the explicit formula of $h$ out as we will not need it. Again, see \cite[Proposition 2.10]{willis2021khovanov} for details.}. 

Note that $T_0$ and $T_1$ are both crossingless tangles, so $\caf(T_0)$ and $\caf(T_1)$ are both chain complexes supported in a single homological grading. Therefore, if $g$ is chain homotopic to $R_1 \circ f \circ R_0^{-1},$ then in fact we have
\[g = R_1\circ f \circ R_0^{-1},\text{ and } \mathcal{C}(T) = \caf(T)\]equal as chain maps and chain complexes, not merely up to homotopy.

We claim that the map $\varphi\coloneqq H$ satisfies the three properties listed above. Properties ii and iii follow directly from the construction of $H$, and Property i follows from the rigidity statement in Lemma \ref{lem:rigidity of tangle invariants}, which implies all homotopy equivalences between $\caf(\Delta^{-1}_{2k}\circ T^*\circ \Delta_{2k}) $ and $\caf(T)$ are chain homotopic as we are working over $\FF=\FF_2$. Note that both $T$ and $\Delta_{2k}^{-1}\circ T^{*} \circ \Delta_{2k}$ are braids, so the concatenation conditions required in the lemma are satisfied.
\end{proof}

Before proving Theorem \ref{thm kappa=id}, recall from Example \ref{exmp unlink kappa} that if $D$ is a crossingless diagram of an unlink (with arbitrary relative position to the axis), then the flip map $\kappa_D\colon\CKh(D) \to \CKh(D)$ is (on the nose, not merely chain homotopic to) the identity map $\idd_{\CKh(D)}$. In other words, Theorem \ref{thm kappa=id} holds in a stronger form for unlinks with crossingless diagrams. The idea of the proof is then to reduce the computation to the case of crossingless diagrams using Lemma \ref{lem algebraic flipping of an elementary tangle}.

 \begin{proof}[Proof of Theorem \ref{thm kappa=id}]
Let $D$ be the plat closure of a braid representing $K$ with $2k$ strands and $n$ crossings. There is a natural $\left\{0,1\right\}^n$-grading given by the cube of resolutions on $\CKh(D)$. In each step of the map $\kappa_D = R_{D^*}\circ \eta_D$ as in Equation (\ref{eqn decompose kappa}), there is a similar $\left\{0,1\right\}^n$-grading defined by only considering the resolutions associated to the $n$ crossings coming from $D$. We call this the \textit{cubical grading}. Note that each of the maps $\eta$, $\Omega$ and $\rho_0$ preserves the cubical grading: it follows from definition for $\eta$ and the fact that Reidemeister moves involved do not affect the $n$ crossings coming from $D$ for $\Omega$ and $\rho_0$. 

It is a priori unclear how the map $\rho$ interacts with the cubical grading as it in general involves Reidemeister III moves that moving these $n$ crossings. Nevertheless, by applying Lemma \ref{lem algebraic flipping of an elementary tangle} to each tangle $T^*_i$ for $i=1,\dots,n$, we can replace $\rho$ by another map $\rho'$, such that 
\begin{enumerate}[i.]
    \item $\rho \simeq \rho'.$
    \item $\rho'$ is filtered with respect to the $\left\{0,1\right\}^n$-grading.
    \item The cubical grading-preserving piece of $\rho'$ is the one induced by a sequence of Reidemeister moves from $\Delta^{-1}_{2k}\circ (T^*_i)_0\circ \Delta_{2k} $ to $(T_i)_0$ realizing the flip cobordism, and from $\Delta^{-1}_{2k}\circ (T^*_i)_1\circ \Delta_{2k} $ to $(T_i)_1 $ respectively, where $(T_i)_{0}$ (resp. $(T_i)_1$) is the $0$(resp. $1$)-resolution of $T_i$, for $i=1,2,\dots,n$.
\end{enumerate}

We now replace $\rho$ by $\rho'$ and define 
\[\kappa' \coloneqq \rho_0\circ \rho' \circ \Omega \circ \eta\colon\CKh(D) \to \CKh(D).\]By properties i and ii above, $\kappa'$ is chain homotopic to $\kappa$ and is filtered with respect to the $\left\{0,1\right\}^n$-grading. Moreover, $\kappa$ preserves the homological grading on $\CKh(D)$, and so does $\kappa'$. Note that the homological grading on $\CKh(D)$ is just the sum of the $\{0,1\}^n$-grading (with some shift that only depends on $D$). We conclude that $\kappa'$ actually \textit{preserves} the cubical grading. In other words, it sends each summand $\CKh(D_v)$ in the resolution cube to the corresponding summand $\CKh(D_v)$ for all $v\in \left\{0,1\right\}^n$. 

By property iii above, the restriction of $\kappa'|_{D_v}\colon\CKh(D_v) \to \CKh(D_v)$ coincides with the flip map $\kappa_{D_v}$ on $D_v$, which is a crossingless diagram of an unlink. By Example \ref{exmp unlink kappa}, $\kappa'|_{D_v}$ is the identity map. Hence, it follows that $\kappa$ is chain homotopic to the identity map on $\CKh(D)$.

Finally, when $D$ is not a plat closure of a braid, an argument similar to the proof of Lemma \ref{lem: rotating by 2pi} (see also \cite[Lemma 3.2]{morrison2022invariants}) shows that $\kappa$ is still chain homotopic to the identity map, which concludes the proof.
 \end{proof}

 \begin{rem}
     The main ingredient of the proof, Lemma \ref{lem algebraic flipping of an elementary tangle}, also works at the level of Bar-Natan's tangle invariants in \cite{bar2005khovanov} by adapting the corresponding arguments as in \cite{willis2021khovanov}*{Corollary~2.14}. The rigidity result is stated using the term \textit{Kh-simple} in \cite{bar2005khovanov}*{Section~8.3}. In particular, one can obtain similar results on the Bar-Natan's deformation of Khovanov homology. See Section \ref{sec:Bar-Natan deformation} for further discussion.
 \end{rem}

As an unfortunate consequence of Theorem \ref{thm kappa=id}, the simplest form of the involutive Khovanov homology does not provide any more information than the ordinary Khovanov homology.

\begin{cor}\label{coro: KhI trivial}
	We have \[\KhI(K;\FF)\cong\Kh(K;\FF)\otimes_\FF \FF[Q]/Q^2\]for all $K$.
\end{cor}

For comparison, in involutive Heegaard Floer homology, the similar relation \[\HFI(Y;\FF)\cong \HF(Y;\FF)\otimes \FF[Q]/Q^2\] holds when $Y$ is an L-space, i.e., for those with minimal possible rank of Heegaard Floer homology. On the other hand, for $Y=\Sigma(2,3,7)$, $\HFI(Y;\FF)$ has rank $4$ whereas $\HF(Y;\FF)$ has rank $3$.

Lipshitz and Sarkar \cite[Proposition 6.5]{lipshitz2014khovanov} used the flip map to prove that their Khovanov stable homotopy type is independent of the choice of (left or right) ladybug matchings. More specifically, they defined a homotopy equivalence \[\mathcal{X}^l(D)\xrightarrow{\eta}\mathcal{X}^r(D^*)\xrightarrow{R}\mathcal{X}^r(D).\]Here, $\mathcal{X}^l(D)$ (resp. $\mathcal{X}^r(D)$) is the Khovanov stable homotopy type constructed from the knot diagram $D$ using left (resp. right) ladybug matchings, $\eta$ is the obvious identification between $\mathcal{X}^l(D)$ and $\mathcal{X}^r(D^*)$, and $R$ is a map induced by a sequence of Reidemeister moves between $D$ and $D^*$, which has showed to be a homotopy equivalence. It is obvious from the construction that the $\eta$ above is a spectral lift of $\eta\colon \CKh(D)\to\CKh(D^*)$. We have the following corollary of Theorem \ref{thm kappa=id}.

\begin{cor}
    The homotopy equivalence $R\circ \eta \colon\mathcal{X}^l(D) \to \mathcal{X}^r(D)$ can be chosen canonically such that it induces the identity map on homology with $\FF_2$ coefficients, i.e., the following diagram commutes:
    \[
    \begin{tikzcd}[row sep = 10mm, column sep =  10mm]
        \mathcal{X}^l(D)\ar[r,"R\circ \eta"] \ar[d,"H_{*}"] & \mathcal{X}^r(D) \ar[d,"H_{*}"]\\
        \Kh(D) \ar[r,"\idd_{\Kh(D)}"] & \Kh(D)
    \end{tikzcd}.
    \]
    
\end{cor}

\section{The half sweep-around map}\label{sec: half sweep around}

In \cite{morrison2022invariants}, to establish the full functoriality of Khovanov homology for links in $S^3$ (instead of in $\mathbb{R}^3$), Morrison, Walker, and Wedrich proved that the sweep-around move induces the identity map on the Khovanov chain complex. Their proof proceeds by showing that the map induced by the front half sweep-around move agrees with that induced by the back half sweep-around move. In this section, using a variation of the previous result on the flip map for $1$-$1$ tangles, we prove that the map induced by the front half sweep-around move (and likewise the back one) is chain homotopic to the identity map after composing with the canonical identification, as described in Theorem \ref{thm alg=top 11 tangle}. Throughout the section, a $1$-$1$ tangle is a vertically presented tangle with $1$ endpoint on the top and bottom respectively.

We first write down the half sweep-around map precisely, and see how it is related to the flip map we have discussed in Section \ref{sec:flip map}. For simplicity, we assume all the $1$-$1$ tangles in this section are presented in the \textit{almost braid closure form} in the sense of \cite[Section 3]{morrison2022invariants}, which is obtained by taking the braid closure of the $k-1$ rightmost strands of a braid word in the braid group $\operatorname{Br}_{k}$. By \cite{morrison2022invariants}*{Lemma~3.1}, every $1$-$1$ tangle can be isotoped into an almost braid closure.

\begin{defn}
    Let $T$ be a $1$-$1$ tangle. Let $r(T)$ (resp. $l(T)$) denote the link obtained by taking the right-handed (resp. left-handed) closure of the $1$-$1$ tangle $T$ respectively. The \textit{half sweep-around map} \[\kappa^{sw}_{l(T)}\colon \CKh(l(T)) \to \CKh(l(T))\] is defined as the composition:
    \[\kappa^{sw}_{l(T)} = R^{sw}_{r(T)} \circ \eta^{sw}_{l(T)}, \]
     where 
     \[\eta^{sw}_{l(T)}\colon \CKh(l(T))\to \CKh(r(T))\] is the canonical identification of $\CKh(l(T))$ with $\CKh(r(T))$, sending each resolution of $l(T)$ to the corresponding one of $r(T)$, and 
     \[R^{sw}_{r(T)}\colon \CKh(r(T))\to \CKh(l(T)) \] is the map induced by a sequence of Reidemeister moves that realizes the isotopy that rotates the rightmost strand to the left from above the projection plane. See Figure \ref{fig:half sweep around} for illustration.
\end{defn}
    \begin{figure}[ht]
	\[{
		\fontsize{8pt}{10pt}\selectfont
		\def\svgscale{1.2}
		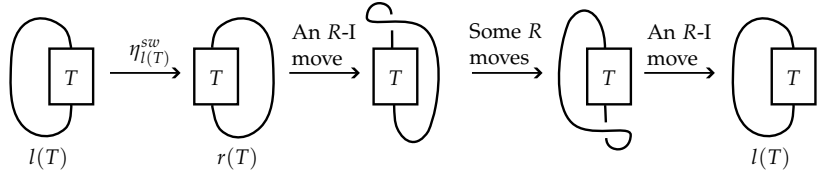
	}\]
	\caption{The half sweep-around map $\kappa^{sw}_{l(T)}$.}
	\label{fig:half sweep around}
\end{figure}
\begin{rem}
    Here we use the \textit{front sweep} to define the half sweep-around map, and one can also use the \textit{back sweep}. \cite[Theorem 3.3]{morrison2022invariants} asserts that these two sequences of Reidemeister moves induce the identical map when $T$ is in the almost braid closure form. Note that our definition of the half sweep-around map is slight different from the one in \cite{morrison2022invariants}: the latter is the $R^{sw}_{r(T)}$ map in our notation.
\end{rem}

Theorem \ref{thm alg=top 11 tangle} is then equivalent to the claim that the map $\kappa^{sw}_{l(T)}$ is chain homotopic to the identity map on $\CKh(l(T))$. The idea of proof is to express the half sweep-around map as the composition of the flip map and another \textit{rolling map} that we have better understanding, and prove the latter is also chain homotopic to the identity map. Roughly speaking, the rolling map for a $1$-$1$ tangle fixes the endpoints and rotates the interior of the tangle for $180$ degrees.

\begin{defn}
    Let $T$ be a $1$-$1$ tangle and $l(T)$ be its left-handed closure. Let $T^*$ denote the $1$-$1$ tangle obtained by rotating $T$ about the vertical axis in the plane passing through the two ends of the tangle $T$, and $l(T^*)$ denote the left-handed closure of $T^*$.  The \textit{rolling map}
    \[\kappa^{roll}_{l(T)}\colon \CKh(l(T)) \to \CKh(l(T))\]
    is defined as the composition 
    \[\kappa^{roll}_{l(T)} = R^{roll}_{l(T^*)} \circ \eta^{roll}_{l(T)},\]
    where \[\eta^{roll}_{l(T)}\colon \CKh(l(T)) \to \CKh(l(T^*))\] is the canonical identification, sending each resolution of $l(T)$ to the corresponding one in $l(T^*)$, and \[R^{roll}_{l(T^*)}\colon \CKh(l(T^*)) \to \CKh(l(T))\] is the map induced by a sequence of Reidemeister moves that realizes the trace of rotating $T^*$ by $\pi t \left(t\in \left[0,1\right]\right)$ about the vertical axis passing through the two ends of the tangle $T$.
 See Figure \ref{fig:rolling} for illustration.
\end{defn}
    \begin{figure}[ht]
	\[{
		\fontsize{8pt}{10pt}\selectfont
		\def\svgscale{1.2}
		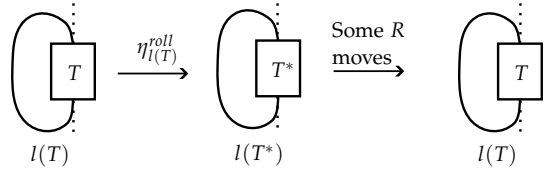
	}\]
	\caption{The rolling map $\kappa^{roll}_{l(T)}$. The dashed line represents the rotational axis.}
	\label{fig:rolling}
\end{figure}

\begin{lem}\label{lem: flip=sw+roll}
    Let $T$ be a $1$-$1$ tangle, and $l(T)$ be its left-handed closure. Let $\kappa_{l(T)}$, $\kappa^{sw}_{l(T)} $ $\kappa^{roll}_{l(T)}$ be the flip map, the half sweep-around map, and the rolling map respectively. Then,
    \[\kappa_{l(T)} \simeq \kappa^{sw}_{l(T)} \circ \kappa^{roll}_{l(T)}.\]
\end{lem}

\begin{proof}
The idea is to realize the Reidemeister moves $R_{l(T)}$ used in the definition of the flip map $\kappa_{l(T)}$ in two steps: first, rotate the $1$-$1$ tangle $T$ by $180$ degrees, and then rotate the closing strand. 

More precisely, note that 
\[(l(T))^* = r(T^*).\]By the discussion above and naturality, we have 
\[\eta^{sw}_{l(T^*)} \circ \eta^{roll}_{l(T)} = \eta_{l(T)}, \,\,\,R^{sw}_{r(T)} \circ R^{roll}_{r(T^*)} \simeq R_{r(T^*)} = R_{(l(T))^*}.\]
Additionally, we also have \[\eta^{sw}_{l(T)} \circ R^{roll}_{l(T^*)} \simeq R^{roll}_{r(T^*)} \circ \eta^{sw}_{l(T^*)}.\]
See Figure \ref{fig: roll and sw commute} below for an illustration, where the horizontal maps are induced by the Reidemeister moves realizing the rolling, and the vertical maps are the canonical identifications. To summarize, we have
\begin{align*}
    \kappa^{sw}_{l(T)} \circ \kappa^{roll}_{l(T)} & = R^{sw}_{r(T)} \circ \eta^{sw}_{l(T)} \circ R^{roll}_{l(T^*)} \circ \eta^{roll}_{l(T)} \\
    & \simeq R^{sw}_{r(T)} \circ R^{roll}_{r(T^*)} \circ \eta^{sw}_{l(T^*)} \circ \eta^{roll}_{l(T)}\\
    & \simeq R_{(l(T))^*} \circ \eta_{l(T)} = \kappa_{l(T)}. 
\end{align*}

\begin{figure}[ht] 
	\[{
		\fontsize{8pt}{10pt}\selectfont
		\def\svgscale{1.2}
		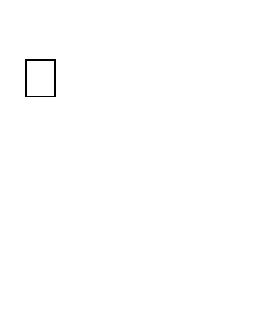
	}\]
	\caption{$\eta^{sw}_{l(T)} \circ R^{roll}_{l(T^*)} \simeq R^{roll}_{r(T^*)} \circ \eta^{sw}_{l(T^*)}.$}
    \label{fig: roll and sw commute}
\end{figure}
\end{proof}

Similar to Theorem \ref{thm kappa=id}, we can prove that the rolling map is also chain homotopic to the identity map.
\begin{prop}
\label{prop:rolling is trivial}
     Let $T$ be a $1$-$1$ tangle, and $l(T)$ be its left-handed closure. Then the rolling map is chain homotopic to the identity map:
     \[\kappa^{roll}_{l(T)} \simeq \idd_{\CKh(l(T))}.\]
\end{prop}
As we have proved that the flip map of $l(T)$ is chain homotopic to the identity map, it follows from Lemma \ref{lem: flip=sw+roll} that the half sweep-around map is also chain homotopic to the identity map, which is another formulation of Theorem \ref{thm alg=top 11 tangle}. 
\setcounter{section}{1}
\setcounter{thm}{6}
\begin{thm}
      Let $T$ be a $1$-$1$ tangle, and $l(T)$ be its left-handed closure. 
      Then, the algebraic identification
      \[\eta^{sw}_{l(T)}\colon \CKh(l(T)) \to \CKh(r(T))\] and the topological identification
      \[R^{sw}_{l(T)}\colon \CKh(l(T)) \to \CKh(r(T)) \]
      are chain homotopic. Equivalently, the composition $ \kappa^{sw}_{l(T)} = R^{sw}_{r(T)}\circ \eta^{sw}_{l(T)}$ is chain homotopic to the identity map $\idd_{\CKh(l(T))}$, where $R^{sw}_{r(T)}$ is the map induced by reversing the Reidemeister moves in the definition of $R^{sw}_{l(T)}$.
\end{thm}
  \setcounter{section}{5}
\setcounter{thm}{6}

\begin{proof}[Proof of Proposition \ref{prop:rolling is trivial}]

Recall that the $1$-$1$ tangle $T$ is assumed to be in almost braid closure form, which is obtained by taking the braid closure of the $k-1$ rightmost strands of a braid word $w$ of length $n$ in the braid group $\Br_{k}$. Then, the link diagram $l(T)$ is of the form
\[l(T) = T_0\circ T_1\circ \cdots \circ T_n\circ T_{n+1},\]
Here, $T_0$ (resp. $T_{n+1}$ ) is the $(0,k)$(resp. $(k,0)$)-tangles of $k$ caps (resp. $k$ cups), closing up the first and second strand, and the $j$-th and the $(2k+3-j)$-th strands, for $j=3,\dots,k+1$, counted from left to right. For each $i=1,\dots,n$,  $T_i$ is the one-crossing tangle representing the $i$-th word in $w$ viewed as an element in $\Br_{2k}$ by including $\Br_{k}$ as a subgroup of $\Br_{2k}$ acting on the second, third, \dots, $(k+1)$-th strands from left to right. See Figure \ref{fig:rolling proof} for an example.

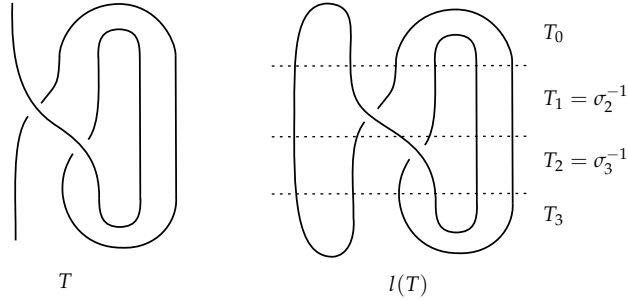
\begin{figure}[ht]
	\[{
		\fontsize{8pt}{10pt}\selectfont
		\def\svgscale{0.9}
		%% Creator: Inkscape 1.2.1 (9c6d41e410, 2022-07-14), www.inkscape.org
%% PDF/EPS/PS + LaTeX output extension by Johan Engelen, 2010
%% Accompanies image file 'rolling_proof.pdf' (pdf, eps, ps)
%%
%% To include the image in your LaTeX document, write
%%   \input{<filename>.pdf_tex}
%%  instead of
%%   \includegraphics{<filename>.pdf}
%% To scale the image, write
%%   \def\svgwidth{<desired width>}
%%   \input{<filename>.pdf_tex}
%%  instead of
%%   \includegraphics[width=<desired width>]{<filename>.pdf}
%%
%% Images with a different path to the parent latex file can
%% be accessed with the `import' package (which may need to be
%% installed) using
%%   \usepackage{import}
%% in the preamble, and then including the image with
%%   \import{<path to file>}{<filename>.pdf_tex}
%% Alternatively, one can specify
%%   \graphicspath{{<path to file>/}}
%% 
%% For more information, please see info/svg-inkscape on CTAN:
%%   http://tug.ctan.org/tex-archive/info/svg-inkscape
%%
\begingroup%
  \makeatletter%
  \providecommand\color[2][]{%
    \errmessage{(Inkscape) Color is used for the text in Inkscape, but the package 'color.sty' is not loaded}%
    \renewcommand\color[2][]{}%
  }%
  \providecommand\transparent[1]{%
    \errmessage{(Inkscape) Transparency is used (non-zero) for the text in Inkscape, but the package 'transparent.sty' is not loaded}%
    \renewcommand\transparent[1]{}%
  }%
  \providecommand\rotatebox[2]{#2}%
  \newcommand*\fsize{\dimexpr\f@size pt\relax}%
  \newcommand*\lineheight[1]{\fontsize{\fsize}{#1\fsize}\selectfont}%
  \ifx\svgwidth\undefined%
    \setlength{\unitlength}{245.16793514bp}%
    \ifx\svgscale\undefined%
      \relax%
    \else%
      \setlength{\unitlength}{\unitlength * \real{\svgscale}}%
    \fi%
  \else%
    \setlength{\unitlength}{\svgwidth}%
  \fi%
  \global\let\svgwidth\undefined%
  \global\let\svgscale\undefined%
  \makeatother%
  \begin{picture}(1,0.49521435)%
    \lineheight{1}%
    \setlength\tabcolsep{0pt}%
    \put(0.07975836,0.00688027){\color[rgb]{0,0,0}\makebox(0,0)[lt]{\lineheight{1.25}\smash{\begin{tabular}[t]{l}$T$\end{tabular}}}}%
    \put(0.645266,0.00300602){\color[rgb]{0,0,0}\makebox(0,0)[lt]{\lineheight{1.25}\smash{\begin{tabular}[t]{l}$l(T)$\end{tabular}}}}%
    \put(0,0){\includegraphics[width=\unitlength,page=1]{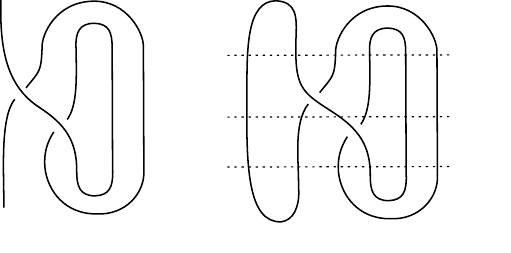}}%
    \put(0.90817428,0.43358556){\color[rgb]{0,0,0}\makebox(0,0)[lt]{\lineheight{1.25}\smash{\begin{tabular}[t]{l}$T_0$\end{tabular}}}}%
    \put(0.90817428,0.32014558){\color[rgb]{0,0,0}\makebox(0,0)[lt]{\lineheight{1.25}\smash{\begin{tabular}[t]{l}$T_1=\sigma_2^{-1}$\end{tabular}}}}%
    \put(0.90817428,0.21350176){\color[rgb]{0,0,0}\makebox(0,0)[lt]{\lineheight{1.25}\smash{\begin{tabular}[t]{l}$T_2=\sigma_3^{-1}$\end{tabular}}}}%
    \put(0.90817428,0.12197016){\color[rgb]{0,0,0}\makebox(0,0)[lt]{\lineheight{1.25}\smash{\begin{tabular}[t]{l}$T_3$\end{tabular}}}}%
  \end{picture}%
\endgroup%

	}\]
	\caption{A $1$-$1$ tangle $T$ in almost braid closure form, and the decomposition of its left closure $l(T)$.}
	\label{fig:rolling proof}
\end{figure}

We can apply a similar construction as in Construction \ref{con: flip map} to realize the map $\kappa^{roll}_{l(T)}$. The only difference is that we now insert positive and negative half twists only on the right $2k-1$ strands, leaving the leftmost strand untouched in the whole process. With an argument similar to the proof of Theorem \ref{thm kappa=id}, we can perturb $\kappa^{sw}_{l(T)}$ to be filtered with respect to the cubical grading, which further deduces that the map preserves the cubical grading and reduces the computation to the case where $T$ is a crossingless $1$-$1$ tangle. Then, as in Example \ref{exmp unlink kappa}, we can further reduce to the computation when $T$ is a $1$-$1$ tangle in Figure \ref{fig:a simple tangle}. An explicit computation similar to Example \ref{exmp unknot kappa} proves that the map $\kappa^{roll}_{l(T)}$ is the identity map in this case.
\begin{figure}[ht]
	\[{
		\fontsize{8pt}{10pt}\selectfont
		\def\svgscale{1.2}
		%% Creator: Inkscape 1.2.1 (9c6d41e410, 2022-07-14), www.inkscape.org
%% PDF/EPS/PS + LaTeX output extension by Johan Engelen, 2010
%% Accompanies image file 'a_simple_tangle.pdf' (pdf, eps, ps)
%%
%% To include the image in your LaTeX document, write
%%   \input{<filename>.pdf_tex}
%%  instead of
%%   \includegraphics{<filename>.pdf}
%% To scale the image, write
%%   \def\svgwidth{<desired width>}
%%   \input{<filename>.pdf_tex}
%%  instead of
%%   \includegraphics[width=<desired width>]{<filename>.pdf}
%%
%% Images with a different path to the parent latex file can
%% be accessed with the `import' package (which may need to be
%% installed) using
%%   \usepackage{import}
%% in the preamble, and then including the image with
%%   \import{<path to file>}{<filename>.pdf_tex}
%% Alternatively, one can specify
%%   \graphicspath{{<path to file>/}}
%% 
%% For more information, please see info/svg-inkscape on CTAN:
%%   http://tug.ctan.org/tex-archive/info/svg-inkscape
%%
\begingroup%
  \makeatletter%
  \providecommand\color[2][]{%
    \errmessage{(Inkscape) Color is used for the text in Inkscape, but the package 'color.sty' is not loaded}%
    \renewcommand\color[2][]{}%
  }%
  \providecommand\transparent[1]{%
    \errmessage{(Inkscape) Transparency is used (non-zero) for the text in Inkscape, but the package 'transparent.sty' is not loaded}%
    \renewcommand\transparent[1]{}%
  }%
  \providecommand\rotatebox[2]{#2}%
  \newcommand*\fsize{\dimexpr\f@size pt\relax}%
  \newcommand*\lineheight[1]{\fontsize{\fsize}{#1\fsize}\selectfont}%
  \ifx\svgwidth\undefined%
    \setlength{\unitlength}{42.290016bp}%
    \ifx\svgscale\undefined%
      \relax%
    \else%
      \setlength{\unitlength}{\unitlength * \real{\svgscale}}%
    \fi%
  \else%
    \setlength{\unitlength}{\svgwidth}%
  \fi%
  \global\let\svgwidth\undefined%
  \global\let\svgscale\undefined%
  \makeatother%
  \begin{picture}(1,0.86593997)%
    \lineheight{1}%
    \setlength\tabcolsep{0pt}%
    \put(0.00169876,0.40189121){\color[rgb]{0,0,0}\makebox(0,0)[lt]{\lineheight{1.25}\smash{\begin{tabular}[t]{l}T =\end{tabular}}}}%
    \put(0,0){\includegraphics[width=\unitlength,page=1]{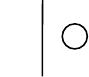}}%
  \end{picture}%
\endgroup%

	}\]
	\caption{The base case of the $1$-$1$ tangle $T$ to check for the rolling map.}
	\label{fig:a simple tangle}
\end{figure}
\end{proof}

One corollary of Theorem \ref{thm alg=top 11 tangle} is an alternative proof of the functoriality of Khovanov homology (with $\FF_2$ coefficients) for links in $S^3$. In fact, as in \cite{morrison2022invariants}, it suffices to check the front sweep and the back sweep induce the same map on Khovanov homology, which follows from Theorem \ref{thm alg=top 11 tangle} as they are both equal to the same algebraic identification $\eta^{sw}_{l(T)}$.

\begin{cor}\label{cor: functoriality in S3}
    The Khovanov homology with $\FF_2$ coefficients associates the identity map to link cobordism represented by a sequence of Reidemeister moves of the form \cite[Equation 1.1]{morrison2022invariants}, and hence, is functorial for links in $S^3$.
\end{cor}

Theorem \ref{thm alg=top braid} now follows from Theorem \ref{thm kappa=id} and Theorem \ref{thm alg=top 11 tangle}. 

\begin{proof}[Proof of Theorem \ref{thm alg=top braid}]
    The braid rolling map can be viewed as a composition of the flip map and $k$ half sweep-around maps, where $k$ is the number of strands. The theorem then follows from the results in Theorems \ref{thm kappa=id} and \ref{thm alg=top 11 tangle}. Alternatively, one can also prove this directly using a filtration argument similar to the proof of Theorem \ref{thm kappa=id}, and we leave the details to the reader.
\end{proof}

\section{Strongly invertible knots}\label{sec: si knots}

In this section, we restrict our attention to strongly invertible knots and prove Theorem \ref{thm strg invertible same}. We also discuss how our result might help compare the invariants from \cite{lobb2021refinement} and \cite{lipshitz2024khovanov}.

Let $(K,\tau)$ be a strongly invertible knot. Let $D_t$ and $D_i$ be a transvergent diagram and an intravergent diagram of $(K,\tau)$, respectively (See Figure \ref{fig two diagrams for si knots} for an example). The involution $\tau$ induces two involutions \[\tau_{\bullet}\colon\CKh(D_\bullet)\to\CKh(D_\bullet)\,(\bullet\in\{t,i\})\] by directly applying the symmetry on each resolution. It has remained open whether these two maps are the same on homology. To compare them, we need to fix an isomorphism between these two ``models'' of the Khovanov homology of $K$, or more precisely, fix a sequence of Reidemeister moves between $D_t$ and $D_i$. We now give the precise statement of Theorem \ref{thm strg invertible same}.

\begin{thm}\label{thm si knots explicit ver}
Let $\varphi\colon \CKh(D_t)\to\CKh(D_i)$ be a chain homotopy equivalence induced by a fixed sequence of Reidemeister moves from $D_t$ to $D_i$. Then the following diagram commutes up to homotopy:

\begin{equation*}
    \begin{tikzcd}
     \CKh(D_t)\ar[r,"\tau_t"] \ar[d,"\varphi"]& \CKh(D_t)\ar[d,"\varphi"] \\
     \CKh(D_i) \ar[r,"\tau_i"]& \CKh(D_i)
\end{tikzcd}.
\end{equation*}
\end{thm}

The idea of the proof is to replace the algebraic involutions with their topological counterparts, which have better functoriality available. 

\begin{proof}
    Notice that the transvergent diagram $D_t$ can be identified (setwise) with its flip diagram $D_t^*$. Therefore, the algebraic involution $\tau_t$ is the same as the $\eta$ map for general knot diagrams, defined in Section \ref{sec:KhI}. Therefore, by Theorem \ref{thm kappa=id}, $\tau_t$ is chain homotopic to a map $R_t\colon\CKh(D_t)\to\CKh(D_t)$ induced by a sequence of Reidemeister moves that realizes the trace cobordism of rotating $K$ about the axis by $\pi$. 

    On the other hand, for the intravergent diagram $D_i$, $\tau_i$ is actually the same as $R_i$, the map induced by the planar isotopy that rotates $D_i$ through the axis by $\pi$. Denote the link cobordism corresponding to the fixed sequence of Reidemeister moves from $D_t$ to $D_i$ by $W$. The compositions $W\circ S_K$ and $S_K\circ W$ are isotopic rel boundary. By the functoriality of Khovanov homology, the following diagram commutes up to homotopy: \begin{equation*}
    \begin{tikzcd}
     \CKh(D_t)\ar[r,"R_t"] \ar[d,"\varphi"]& \CKh(D_t)\ar[d,"\varphi"] \\
     \CKh(D_i) \ar[r,"R_i"]& \CKh(D_i)
\end{tikzcd},
\end{equation*}which completes the proof.
\end{proof}

 We expect Theorem \ref{thm si knots explicit ver} can be strengthened to a comparison result between the equivariant Khovanov homology from two types of diagrams. In general, an equivariant version of the invariant is desirable when the topological objects carry additional symmetries. In the context of Khovanov homology, Lipshitz and Sarkar \cite{lipshitz2024khovanov} constructed a $(\ZZ)$-equivariant Khovanov stable homotopy type for strongly invertible knots with intravergent diagrams using the machinery of Stoffregen--Zhang \cite{stoffregen2024localization}. This allows us to discuss equivariant versions of Khovanov homology for strongly invertible knots with intravergent diagrams.

For transvergent diagrams, while an equivariant stable homotopy lift of Khovanov homology seems intractable at the moment, one can still consider the Borel complex \begin{equation}\label{eqn: Borel for transvergent}
    \CKh^*_{\text{Borel}}(D_t)\coloneqq \left(\CKh(D_t)\otimes\FF[Q],d_{\Kh}+Q(1+\sigma_t)\right),
\end{equation}where $Q$ is a formal variable of bigrading $(1,0)$. The cohomology of this complex should be thought of as a (Borel) equivariant Khovanov homology\footnote{A general chain model for $(\ZZ)$-equivariant Borel cohomology would have components from higher homotopies in the differential; in our case, they all vanish as $\sigma_t$ is an honest (not just homotopy) involution.} of the strongly invertible knot. This Borel construction has recently been studied (in particular, proved to be an invariant of the strongly invertible knot type) in \cite{borodzik2025khovanov}. It specifies to Lobb--Watson's refinement of Khovanov homology \cite{lobb2021refinement} and Sano's involutive Khovanov homology for strongly invertible knots \cite{sano2024involutive} when setting $Q=1$ and $Q^2=0$ in (\ref{eqn: Borel for transvergent}) respectively.

We conjecture that two versions of equivariant Khovanov homology for strongly invertible knots from \cite{lipshitz2024khovanov} and \cite{borodzik2025khovanov} coincide.

\begin{conj}\label{prop: equivariant si same}
    The Borel equivariant Khovanov chain complexes from the intravergent diagram and the transvergent diagram of a strongly invertible knot are quasi-isomorphic. 
\end{conj}

The missing ingredient in proving Conjecture \ref{prop: equivariant si same} is higher order naturality of Khovanov homology, as we now explain. After fixing a chain homotopy equivalence $\varphi\colon \CKh(D_t)\to\CKh(D_i)$ induced by a sequence of Reidemeister moves, we can add a chain homotopy \[H\colon \CKh(D_t)\to Q\cdot\CKh(D_i)\]to make the homotopy commutative diagram in Theorem \ref{thm si knots explicit ver} commute on the nose. However, the map $\Phi\colon \CKh^*_{\text{Borel}}(D_t)\to\CKh^*_{\text{Borel}}(D_i)$, given by the sum of the vertical map $\varphi$ and the homotopy $H$ in the following diagram, is not necessarily a chain map:\begin{equation*}
    \begin{tikzcd}[row sep=3em, column sep = 5em]
     \CKh(D_t)\ar[r,"Q(1+\tau_t)"] \ar[d,"\varphi"] \ar[rd,"H"]& Q\CKh(D_t)\ar[d,"\varphi"]\ar[r,"Q(1+\tau_t)"] \ar[rd,"H"]& Q^2\CKh(D_t)\ar[d,"\varphi"]\ar[r] & \dots\\
     \CKh(D_i) \ar[r,"Q(1+\tau_i)"]& Q\CKh(D_i) \ar[r,"Q(1+\tau_t)"]& Q^2\CKh(D_i) \ar[r]&\dots
\end{tikzcd}.
\end{equation*}The issue is the homotopy $H$ may not commute with maps induced by Reidemeister moves. In fact, it is more reasonable to expect they should commute up to homotopy, which already requires the (currently unavailable) second order naturality of Khovanov homology. Even under this assumption, one still needs to worry about whether the new homotopy that witnesses the commutativity of the previous diagram commutes (up to homotopy) with the differentials, and so on. In summary, to construct the desired quasi-isomorphism, it seems to require a lift of Khovanov homology as an $(\infty,1)$-functor. 

One can also consider the Tate complex \[ \CKh^*_{\text{Tate}}(D_t)\coloneqq \left(\CKh(D_t)\otimes\FF[Q,Q^{-1}],d_{\Kh}+Q(1+\sigma_t)\right).\]As a double complex with the horizontal differential $Q(1+\sigma_t)$ and vertical differential $d_{\Kh}$, it gives rise to two spectral sequences. Following the notation in \cite{lipshitz2016noncommutative}, we denote the one which first computes the homology with respect to the differential $d_{\Kh}$ by $^{vh}E$, and the other one by $^{hv}E$. When setting $Q=1$, they should specify to the $\caf$ and $\cag$ spectral sequences of Lobb--Watson \cite[Theorem 1.1]{lobb2021refinement}. 

It would be interesting to compare the Tate complexes from transvergent and intravergent diagrams. Conjecture \ref{prop: equivariant si same} would imply that two $^{vh}E$ spectral sequences coincide. However, the conjectural quasi-isomorphism there is not filtered with respect to the vertical filtration, so it would not give the desired quasi-isomorphism between the $E^3$ pages. See Figure \ref{fig: tate ss} for illustration.

\begin{figure}[ht]
\centering
\begin{tikzcd}[row sep=1em, column sep=1em]
  &  Q^{l}\CKh^{k+1}(D_i) \arrow[rr] \arrow[from=dd] & & Q^{l+1}\CKh^{k+1}(D_i)   \\
 Q^{l}\CKh^{k+1}(D_t)\arrow[rr, crossing over]\arrow[ur,"\varphi"] \arrow[drrr, dashed, crossing over] & & Q^{l+1}\CKh^{k+1}(D_t) \arrow[ur,"\varphi"']& \\
  & Q^{l}\CKh^{k}(D_i) \arrow[rr]  & & Q^{l+1}\CKh^{k}(D_i) \arrow[uu] \\
 Q^{l}\CKh^{k}(D_t) \arrow[uu]\arrow[rr]\arrow[ur,"\varphi"] & & Q^{l+1}\CKh^{k}(D_t)\arrow[ur,"\varphi"']\arrow[uu, crossing over]
\end{tikzcd}
\caption{Part of the comparison map between two Tate complexes. The vertical arrows are ordinary Khovanov differentials, and the horizontal arrows are Tate differentials $Q(1+\tau_t)$ (or $Q(1+\tau_i)$). The dashed arrow is the homotopy $H$ defined above, which respects the horizontal filtration but not the vertical filtration.}
\label{fig: tate ss}
\end{figure}
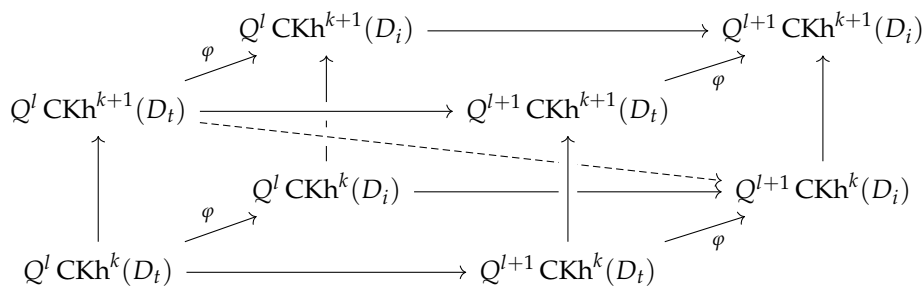

\begin{quest}
    Do the $^{hv}E$ spectral sequences from the Tate complexes of transvergent and intravergent diagrams coincide with each other?
\end{quest}

A positive answer to this question would provide an approach to \cite[Question 6.5]{lobb2021refinement} that if the $\cag$ spectral sequence collapses at the $E^3$ page, and hence prove a Smith-type inequality for Khovanov homology of the transvergent diagram. Again, this seems to require progress on higher order naturality of Khovanov homology.

\section{Discussion}\label{sec: discussions}

\subsection{Comparison to involutive Heegaard Floer homology}

As previously alluded, many constructions appearing in Section \ref{sec:flip map} turn out to be formally similar to the ones in the work of Alishahi--Truong--Zhang \cite{alishahi2023khovanov}, where they used bordered Floer techniques to establish a spectral sequence relating (a truncated version of) the reduced Bar-Natan homology of a knot and the involutive Heegaard Floer homology of its branched double cover. In this subsection, we discuss the similarities and relationship between these two theories in further detail.

One of the key observations in the proof of Theorem \ref{thm kappa=id} is to realize the flip cobordism by half twists. This should be compared to the presence of the Auroux--Zarev piece \cite{auroux2010fukaya,zarev2010joining} in the involutive Heegaard Floer picture. It serves as an interpolating piece that swaps $\alpha$ and $\beta$ curves, which is exactly what the involutive map $\iota$ on Heegaard Floer chain complexes does. It is therefore reasonable to believe the branched double cover of a half twist is related to the corresponding Auroux--Zarev piece, cf. \cite[Proposition 4.2]{lipshitz2011heegaard}. One issue is that the branched double cover of a half twist is $\alpha$-$\alpha$-bordered, while the Auroux--Zarev piece is an $\alpha$-$\beta$-bordered.

We can also compare the proofs of Theorem \ref{thm kappa=id} and \cite[Theorem 4.1]{alishahi2023khovanov}. The maps in (\ref{eqn decompose kappa}) are formally identical to \cite[Equation 3]{alishahi2023khovanov}. We are intentionally using similar letters to denote them: $\eta$ is the canonical (algebraic) identification between Khovanov tangle (resp. bordered Floer) invariants, $\Omega$ is the chain homotopic equivalence that introduces modules of pairs of half twists (resp. Auroux--Zarev pieces), and $\rho$ (resp. $\Psi$) is the map induced by Reidemeister moves (resp. Heegaard moves). Lemma \ref{lem algebraic flipping of an elementary tangle} can be thought as a Khovanov analog of \cite[Lemma 7.2]{hendricks2019involutive}, and the rigidity result used therein should be compared to \cite[Section 4]{hendricks2019involutive},  albeit the latter handles substantially more technical challenges. Finally, while both $\kappa$ and $\iota$ are chain homotopic to maps that are filtered with respect to the cubical grading, the homological grading is related to the cubical grading in Khovanov side, which differs from the Heegaard Floer setting and makes the triviality argument only valid in our setting.

It is still natural to ask if involutive Khovanov homology is related to Heegaard Floer homology through the Ozsv\'ath--Szab\'o's branched double cover spectral sequence \cite{ozsvath2005heegaard,lipshitz2014bordered,lipshitz2016bordered}. 

\begin{quest}
    How does the flip map on Khovanov homology interact with the Ozsv\'ath--Szab\'o spectral sequence?
\end{quest}

\subsection{Variants of Khovanov homology}
\label{sec:Bar-Natan deformation}
It would be interesting to generalize the results in this paper to other variants of Khovanov homology. In particular, we expect the key technical argument, Lemma \ref{lem algebraic flipping of an elementary tangle}, remains valid with some mild modifications. In this case, the flip map would again be completely determined by its behavior on unlinks. However, the flip map is generally not completely trivial on unlinks, as the following example suggests.

Consider the (graded) Bar-Natan homology with $\FF_2$ coefficients \cite{bar2005khovanov}. It assigns an $\FF[H]$-module $\BN(K)$ for a knot $K$ from the Frobenius algebra \[A_{\text{BN}}=\FF[H][X]/(X^2=HX).\] Denote $X + H\cdot1$ by $Y$. Let $U$ be the unknot. Using the sequence of Reidemeister moves illustrated in Figure \ref{fig unknot kappa}, one can compute that $\kappa_*\colon\BN(U)\to\BN(U)$ is given by $\kappa_*(1)=1,\,\kappa_*(X)=Y$. It is easy to check the $\FF[H]$-linear map \[\sigma\colon A_{\text{BN}}\to A_{\text{BN}},\,\sigma(1)=1,\,\sigma(X)=Y \text{ (and hence } \sigma(Y)=X)\]is an automorphism of Frobenius algebra, and it gives an involution on Bar-Natan chain complexes by applying $\sigma$ to all components of generators. In the spirit of Theorem \ref{thm kappa=id}, we have the following proposition. 

\begin{prop}\label{prop: flip for BN}
     Let $D$ be a diagram for a knot $K\subset S^3$. The flip map $\kappa\colon\CBN(D)\to\CBN(D)$ is  chain homotopic to the involution induced by the automorphism $\sigma$.
\end{prop}

\begin{proof}
    Lemma \ref{lem algebraic flipping of an elementary tangle} remains valid with Bar-Natan's tangle invariants formalism. The proof then proceeds similarly to Theorem \ref{thm kappa=id} with a few minor modifications. First, while the perturbed map $\kappa'$ preserves the $\{0,1\}^n$-grading, it is not the identity map but rather the map induced by $\sigma$, as this is the case for unlinks. Second, after establishing the claim for plat closures of braids, the general case follows the fact that the following diagram commutes on the nose:
    \begin{equation*}
    \begin{tikzcd}
     \CBN(D)\ar[r,"\sigma"] \ar[d,"\varphi"]& \CBN(D)\ar[d,"\varphi"] \\
     \CBN(D') \ar[r,"\sigma"]& \CBN(D')
    \end{tikzcd},
    \end{equation*}where $D$ and $D'$ are diagrams of $K$, and $\varphi$ is induced by a sequence of Reidemeister moves between $D$ and $D'$. In the proof of Theorem \ref{thm kappa=id}, a similar claim was implicit in the last paragraph of the proof, where the identity map obviously commutes with $\varphi$. In the Bar-Natan homology setting, $\sigma$ commutes with all elementary cobordisms, as it is an automorphism of Frobenius algebra. Therefore, it commutes with all cobordisms involved when performing Reidemeister moves (cf. \cite[Section 4]{bar2005khovanov}), which completes the proof.
\end{proof}

While Proposition \ref{prop: flip for BN} asserts that $\kappa$ is governed by the intrinsic structure of the Frobenius algebra, one can still define \textit{involutive Bar-Natan complex} $ \CBNI(D)$ of a knot diagram $D$ as
\[\CBNI(D) := \Cone\left(\CBN(D) \xrightarrow{Q(\idd+\sigma)}Q\CBN(D)\right),\]
where $Q$ is a formal variable with quantum grading $0$ and homological grading $1$. Unfortunately, it again turns out that involutive Bar-Natan homology carries no further information, as indicated by the following proposition, which was kindly communicated to us by Taketo Sano.

\begin{prop}[Sano]
    For a knot diagram $D$, we have the following graded chain homotopy equivalence as a chain complex of $\FF[H]$-modules:
    \[\CBNI(D) \simeq \widetilde{\CBN}(D) \oplus Q \widetilde{\CBN}(D) \oplus Q\widetilde{CKh}(D),\] where $\widetilde{\CBN}(D) $ is the reduced Bar-Natan complex, and $\widetilde{CKh}(D)$ is the reduced Khovanov complex.
    \label{prop:BNI trivial}
\end{prop}
\begin{proof}
    Choose a base point on $D$, and let $C_X$ (resp. $C_Y$) be the subcomplex of $\CBN(D)$ such that the component of the state on the marked circle in each resolution equal $X$ (resp. $Y$). The base point also endows $\CBN(D)$ an $A_{\text{BN}}$-module structure, where the action is given by merging an unknot to $D$ at the base point.  Define the map $\upsilon\colon\CBN(D) \to \CBN(D)$ by 
\[\upsilon\coloneqq H^{-1}(\idd+\sigma),\]
(in particular, $\idd + \sigma$ is divisible by $H$ cf. \cite[Lemma 2.2]{khovanov2025symmetries}). By \cite[Proposition 2.9, Remark 2.10]{khovanov2025symmetries}, we have two splitting short exact sequences of chain complexes of $\FF[H]$-modules:
\[
\begin{tikzcd}
     0\ar[r] & C_X\ar[r] & \CBN(D) \ar[r, "\cdot Y"] & C_Y \ar[r] \ar[l,bend left=30, end anchor = south east, "\upsilon"]&0,  \\
      0\ar[r] & C_Y\ar[r] & \CBN(D) \ar[r, "\cdot X"] & C_X \ar[r] \ar[l,bend left, end anchor = south east,"\upsilon"]&0 ,
    \end{tikzcd}
\]
where $\cdot X$ and $\cdot Y$ are multiplication by $X$ and $Y$ at the base point respectively. Hence, we have direct sum splittings\begin{equation}\label{eqn:CBN splits}
    \CBN(D) \cong C_X \oplus \upsilon (C_Y )\cong C_Y \oplus \upsilon (C_X).
\end{equation}
Note the following relations between the maps:
\[\idd +\sigma = H\upsilon, \quad \upsilon^2=0.\]
Then we can rewrite the involutive Bar-Natan chain complex $\CBNI(D)$ using the splittings in (\ref{eqn:CBN splits}) as\[\CBNI(D) \cong \Cone\left(C_X \oplus \upsilon(C_Y) \xrightarrow{\left(\begin{smallmatrix} 0 & 0\\ QHv & 0 \end{smallmatrix}\right)} QC_Y \oplus Q\upsilon(C_X) \right),\]
Hence, as $\FF[H]$-modules, we have
\[\CBNI(D) \cong \upsilon(C_Y) \oplus QC_Y \oplus \Cone \left(C_X \xrightarrow{QHv}Qv(C_X)\right),\]
where $\upsilon(C_Y)$ and $C_Y$ are both isomorphic to the reduced Bar-Natan chain complex \cite{WigdersonBarNatan}, and the rest direct summand is chain homotopic to the reduced Khovanov complex, as we mod out the image of multiplication by $H$.
\end{proof}

Another direction is to extend Theorem \ref{thm kappa=id} to $\Z$-coefficients. It is observed that the flip map is also not the identity map on unlinks for $\Z$-coefficients (cf. \cite[Remark 3.8]{lipshitz2022mixed}). However, this observation seems to be a bug instead of a feature of the ordinary Khovanov cobordism maps. To correct the signs, it would be better to switch to the language of $\mathfrak{gl}_2$ webs and foams \cite{blanchet2010oriented}, which naturally leads to a similar question to Khovanov--Rozansky $\mathfrak{sl}_N$ homology \cite{khovanov2008matrix,ehrig2018functoriality}. We expect the flip map is always identity in this setting, and we plan to explore this in future work. 
	
\bibliographystyle{amsalpha}
\bibliography{ref}

\end{document}